\documentclass[a4paper, reqno, 11pt]{amsart}
\usepackage{color}
\makeatletter

\@addtoreset{equation}{section}
\makeatother
\usepackage{setspace}
\usepackage{xcolor}
\usepackage{here}
\usepackage{graphicx}
\usepackage{float}

\usepackage[T1]{fontenc}
\everymath{\displaystyle}
\usepackage{amsmath}
\usepackage{amssymb}
\usepackage{latexsym}
\usepackage{amsthm}
\usepackage[hidelinks,hypertexnames=false]{hyperref}
\usepackage{mathtools}
\usepackage{mathrsfs}
\usepackage{tikz}
\usetikzlibrary{arrows.meta}
\newtheorem{Thm}{Theorem}[section]

\newtheorem{Lem}[Thm]{Lemma}
\newtheorem{Prop}[Thm]{Proposition}

\newtheorem{Cor}[Thm]{Corollary}

\theoremstyle{definition}

\newtheorem{Def}[Thm]{Definition}
\newtheorem{Rem}[Thm]{Remark}
\newtheorem{Exa}[Thm]{Example}

\DeclareMathOperator{\Lip}{Lip}
\DeclareMathOperator{\dist}{dist}
\DeclareMathOperator{\caplog}{cap}
\DeclareMathOperator{\esssup}{ess\,sup}
\newcommand{\A}{\mathcal A}
\newcommand{\Leb}{\mathcal{L}^1}
\newcommand{\Reg}{\mathrm{Reg}}
\newcommand{\Ball}[2]{B(#2,#1)}
\newcommand{\id}{\mathrm{id}}

\begin{document}

\begin{abstract}
We consider probability measures whose distribution functions solve conjugate equations between two compatible systems of increasing maps on the unit interval.
We give sufficient conditions for Ullman--Stahl--Totik regularity in terms of finite compositions of the source maps and a lower mass estimate which holds at Lebesgue almost every point. 
Our result substantially extends a regularity theorem of Mantica and Totik by allowing nonlinear target maps and replacing their assumptions on convexity, concavity, and displacement near fixed points. 
The proof uses the mass criterion of Stahl and Totik from logarithmic potential theory.
We give further sufficient conditions involving uniform mesh bounds, sets where derivatives equal one, and endpoint drift.
Examples distinguish these conditions and show that the uniform version of the lower mass estimate can fail for regular conjugacy measures.
\end{abstract}

\title{Ullman--Stahl--Totik regularity for conjugacy measures}
\author{Kazuki Okamura}
\date{\today}

\address{Department of Mathematics, Faculty of Science, Shizuoka University, 836 Ohya, Suruga-ku, Shizuoka 422-8529, Japan}
\email{okamura.kazuki@shizuoka.ac.jp}

\maketitle

\section{Introduction}\label{sec:introduction}

The asymptotic behavior of orthogonal polynomials is closely related to the distribution of mass of the underlying measure.
Ullman--Stahl--Totik regularity describes measures for which the $m$th root of the $L^2$ norm of the monic orthogonal polynomial of degree $m$ tends to the logarithmic capacity of the support.
For a finite positive Borel measure with support $[0,1]$, regularity means that this limit is $1/4$. 
Regularity also implies that the normalized zero-counting measures converge weakly to the arcsine measure and that the root asymptotics of the orthonormal polynomials outside $[0,1]$ are determined by the Green function of its complement; see \cite[Theorems~1.7 and~1.10]{Simon}.
These connections make regularity useful in approximation theory and in the spectral theory of Jacobi matrices \cite{Saff,Simon}.

The theory has its origins in the work of Ullman \cite{Ullman} and was developed systematically by Stahl and Totik \cite{StahlTotik}.
An important part of this theory consists of sufficient conditions for regularity expressed directly in terms of local masses.
We use the criterion $\lambda^*$, stated in \cite[Theorem~4.2.7]{StahlTotik} and \cite[Criterion~1.1]{ManticaTotik}.
Roughly speaking, this criterion requires that the masses of small intervals should not decay too rapidly on sets whose Lebesgue measures tend to the full measure of the support.
Therefore it allows regularity to be proved from local lower mass estimates, without computing the associated orthogonal polynomials.

Invariant measures of iterated function systems provide a natural setting for these criteria.
Kr{\"u}ger and Simon \cite{KrugerSimon} proved UST regularity for the classical Cantor measure and related measures.
For Minkowski's question mark measure, Dresse and Van Assche \cite{DresseVanAssche} investigated orthogonal polynomials numerically and asked whether the measure is UST regular and whether it belongs to the Nevai class.
Mantica \cite{Mantica2017} obtained further numerical evidence and formulated the regularity conjecture.
Mantica and Totik \cite{ManticaTotik} proved regularity of the question mark measure and of the measure associated with the inverse question mark function.
They also established regularity for more general classes of iterated function systems.
In particular, \cite[Theorem~8.1]{ManticaTotik} deals with increasing weakly contractive maps whose image intervals cover $[0,1]$ and have disjoint interiors.
This result permits neutral fixed points under prescribed convexity or concavity and quadratic lower bounds on the displacement near those fixed points. 
They ask whether these assumptions can be omitted; see \cite[Remark~8.6]{ManticaTotik}.

In this paper, we study regularity for probability measures defined by conjugate equations between two compatible systems of increasing maps on $[0,1]$.
For the source maps $\{f_i\}_i$ and the target maps $\{g_i\}_i$, an increasing homeomorphism $\varphi$ satisfying $\varphi\circ f_i=g_i\circ\varphi$ defines a probability measure $\mu$ by $\mu([0,x])=\varphi(x)$.
This framework includes de Rham's functional equations \cite{deRham} and several classical singular functions.
Zdun \cite{Zdun2001} studied existence, uniqueness, and continuity of solutions for compatible systems of conjugate equations.
Serpa and Buescu \cite{SerpaBuescu2017} developed a general framework for systems of iterative functional equations between topological and metric spaces.

In this framework, we extend the interval result of Mantica and Totik \cite[Theorem~8.1]{ManticaTotik} in two directions. 
First, the balanced measures with constant probability weights considered there correspond to affine target maps $\{g_i\}_i$ in the conjugate equation.
Here the target maps $\{g_i\}_i$ may be nonlinear, provided that they have vanishing mesh and their difference quotients have a uniform positive lower bound. 
Second, we replace Mantica and Totik's assumptions on convexity, concavity, and displacement near fixed points by conditions on finite compositions of the source maps $\{f_i\}_i$ and on local masses. 
The source maps $\{f_i\}_i$ are assumed to be $1$-Lipschitz and to have vanishing mesh.
The main theorem requires no differentiability assumption on either system.
In particular, our criterion allows source maps which preserve distances on nondegenerate subintervals, as well as systems for which a uniform mesh bound of order $1/n$ fails.

The main criterion separates a condition on the source maps $\{f_i\}_i$ from a condition on the measure.
Condition $(R)$ requires that, for some fixed block length, the set of points admitting an address with no strictly contractive block has Lebesgue measure zero.
Condition $(W)$ requires a lower mass bound of the form $\exp(-C/r)$, with a common constant $C$, at Lebesgue almost every point, and the radius threshold may depend on the point.
Theorem~\ref{thm:main} shows that these two conditions imply $\lambda^*$ and hence UST regularity.
The proof uses the Stahl--Totik mass criterion from logarithmic potential theory. 
In order to verify it, we first use $(R)$ to prove that the Lipschitz constants of successive prefix compositions tend to zero outside a Lebesgue null set, and then transfer $(W)$ through the conjugate equation and obtain $-r\log\mu([x-r,x+r]\cap[0,1])\to0$ as $r\downarrow0$ at Lebesgue almost every $x$.
This gives a method based on finite compositions and local mass estimates for establishing regularity in logarithmic potential theory.

We also give sufficient conditions for $(R)$ and $(W)$.
A uniform source mesh bound of order $1/n$, denoted by $(M)$, implies the stronger uniform mass condition $(UW)$.
For $C^1$ source maps $\{f_i\}_i$, finiteness of the derivative equality sets for all compositions of some fixed length implies $(R)$.
For $C^2$ source maps $\{f_i\}_i$, condition $(M)$ is sufficient for UST regularity under the standing assumptions.
Further criteria are given in terms of endpoint drift and quadratic contraction estimates for finite compositions. 
Proposition~\ref{prop:MT} shows that the assumptions of \cite[Theorem~8.1]{ManticaTotik} imply our sufficient conditions, so that their theorem is recovered in its full stated setting. 
The examples show that the extension includes cases outside those assumptions.
In particular, Example~\ref{ex:R-not-F} has intervals on which a source map has derivative one, whereas Example~\ref{ex:W-not-UW} satisfies the main criterion even though $(M)$ fails and $(UW)$ fails for affine target maps $\{g_i\}_i$.
As a special case, we obtain UST regularity for the linear fractional systems considered in \cite{Okamura2026}, where a uniform lower mass estimate was already obtained in \cite[Proposition~2.13]{Okamura2026}.

A related potential-theoretic approach to de Rham equations was recently given by Arcozzi \cite{Arcozzi}.
He represents certain M\"obius de Rham functions as normalized cumulative dyadic capacities and refines estimates for the size of sets carrying the associated measures by using dyadic Riesz capacities.
These results concern capacitary representations and singularity, while our question is UST regularity and the resulting asymptotic behavior of orthogonal polynomials.

The structure of this paper is as follows.
We give the framework and state the main theorem below.
Section~\ref{sec:proof} contains its proof.
Sections~\ref{sec:finite} and~\ref{sec:mesh} give sufficient conditions based on finite derivative equality sets and uniform mesh bounds.
Section~\ref{sec:MT} relates the criterion to the assumptions of Mantica and Totik, and Section~\ref{sec:endpoints} gives endpoint criteria and their application to linear fractional systems.
Section~\ref{sec:implications} summarizes the implications between the conditions according to the differentiability of the source maps $\{f_i\}_i$.
Section~\ref{sec:examples} gives examples distinguishing the conditions, including a general UST-regular measure for which $(W)$ fails.
Finally, Section~\ref{sec:cesaro-nevai} applies Simon's theorem to obtain the Ces\`aro--Nevai condition for the recurrence coefficients.

\subsection{Framework and main result}\label{subsec:framework}

Let $N\ge2$ and $\A \coloneqq \{0,\ldots,N-1\}$. 
We first define the compatibility condition for a family of interval maps.

\begin{Def}\label{def:compatible}
A family $\{h_i\}_{i\in\A}$ of strictly increasing continuous maps from $[0,1]$ into itself is called compatible if
\begin{equation*}
 h_0(0) = 0,\qquad h_i(1) = h_{i+1}(0), \quad 0\le i<N-1,\qquad h_{N-1}(1) = 1.
\end{equation*}
\end{Def}

For $\sigma = (\sigma_1,\ldots,\sigma_n)\in\A^n$, let $h_\sigma \coloneqq h_{\sigma_1}\circ\cdots\circ h_{\sigma_n}$, $h_\emptyset \coloneqq \id$, and
\[ \delta_n(h) \coloneqq \max_{\sigma\in\A^n} h_{\sigma}(1) - h_{\sigma}(0), \quad D_h \coloneqq \left\{h_\sigma(e) \colon \sigma\in\A^*,\ e\in\{0,1\} \right\}.\]
Here $\A^*$ includes the empty word. 
For a finite word $\sigma\in\A^n$, write $|\sigma|=n$ for its length, with $|\emptyset|=0$. 

We say that the system has {\it vanishing mesh} if $\lim_{n \to \infty} \delta_n(h) = 0$.

Unless explicitly stated otherwise, the source maps $\{f_i\}_i$ form a compatible system and satisfy
\begin{equation}\label{eq:basic-f}
 \Lip(f_i)\le1, \quad i\in\A,\qquad \lim_{n\to\infty}\delta_n(f)=0.
\end{equation}
The target maps $\{g_i\}_i$ also form a compatible system and satisfy
\begin{equation}\label{eq:basic-g}
 c_g\coloneqq\min_{i\in\A}\inf_{0\le u<v\le1}\frac{g_i(v)-g_i(u)}{v-u}>0,
 \qquad \lim_{n\to\infty}\delta_n(g)=0.
\end{equation}
The constant $c_g$ depends only on the target maps $\{g_i\}_i$.
Compatibility gives $0<c_g\le g_i(1)-g_i(0)<1$ for every $i\in\A$.
The first condition in \eqref{eq:basic-g} is equivalent to Lipschitz continuity of each inverse map $g_i^{-1}\colon g_i([0,1])\to[0,1]$.
In fact, $\Lip(g_i^{-1})\le c_g^{-1}$ for every $i\in\A$.
No differentiability assumption is imposed on the target maps $\{g_i\}_i$.
The lower bound holds, in particular, when every $g_i$ is $C^1$ with positive derivative on $[0,1]$.
Whenever derivatives are used, derivatives at $0$ and $1$ are one-sided.

Both systems have vanishing mesh by \eqref{eq:basic-f} and \eqref{eq:basic-g}.
Consequently, there exists a unique increasing homeomorphism $\varphi \colon [0,1] \to [0,1]$ satisfying 
\begin{equation}\label{eq:conjugacy}
 \varphi \circ f_i = g_i \circ \varphi, \qquad i\in\A.
\end{equation}
Existence and uniqueness of $\varphi$ follow from \cite[Lemma~2.3]{Okamura2026}.
Indeed, vanishing mesh implies that $D_f$ and $D_g$ are dense in $[0,1]$, so both compatible families are $D$-systems in the terminology of \cite[Subsection~2.1]{Okamura2026}.
The resulting continuous strictly increasing solution fixes $0$ and $1$, and hence it is a homeomorphism of $[0,1]$ onto itself.
Define the probability measure $\mu=\mu_\varphi$ by $\mu([0,x])=\varphi(x)$ for $x \in [0,1]$.
The measure $\mu$ is nonatomic and has support $[0,1]$.
We use these systems, the conjugacy $\varphi$, and the measure $\mu$ throughout the following regularity results.

Write $\Ball r x \coloneqq [x-r,x+r] \cap [0,1]$ for $x \in [0,1]$ and $r \ge 0$.
We denote the Lebesgue measure by $\Leb$. 
We also write $|I|$ for the length $\Leb(I)$ of an interval $I$. 
We recall the mass criterion $\lambda^*$.

\begin{Def}\label{def:lambda}
A Borel probability measure $\mu$ with support $[0,1]$ satisfies condition $\lambda^*$ if, for every $\eta>0$,
\begin{equation}\label{eq:lambda}
 \lim_{s \to \infty} \Leb\left( \left\{x\in[0,1]:\mu(\Ball{1/s}{x})\ge \exp(-\eta s) \right\} \right) = 1.
\end{equation}
\end{Def}

By \cite[Criterion~1.1]{ManticaTotik}, condition $\lambda^*$ implies Ullman--Stahl--Totik (UST) regularity, denoted by $\mu\in\Reg([0,1])$.
This regularity concerns the $n$th-root asymptotics of orthogonal polynomials and its precise formulation is recalled in Subsection~\ref{subsec:regular-without-W}.

For a nonatomic probability measure $\mu$ with support $[0,1]$, let
\begin{equation*}
 A_\mu(y) \coloneqq \limsup_{r\downarrow0}  \left[-r\log\mu(\Ball r y)\right]\in[0,\infty].
\end{equation*}
This function is Borel measurable, since the ball mass is positive and continuous in $(y,r)$ for $r>0$, and the upper limit may be taken over rational radii.

We introduce two lower mass conditions, one uniform and the other holding at Lebesgue almost every point.

\begin{Def}\label{def:mass}
For a nonatomic probability measure $\mu$ on $[0,1]$, we consider the following two conditions. \\
(i) We say that $(UW)$ holds for $\mu$  if there exists $C_1 > 0$ such that
\begin{equation*}
 \mu(\Ball t y)\ge\exp(-C_1/t),
\end{equation*}
for every $(y,t)$ such that $0<y<1,\ 0<t\le\min\{y,1-y\}$. \\
(ii) We say that $(W)$ holds for $\mu$  if
\begin{equation*}
 \esssup_{y\in(0,1)} A_\mu(y) < \infty,
\end{equation*}
where the essential supremum is with respect to Lebesgue measure.
\end{Def}

Clearly, $(UW)$ implies $(W)$.
Condition $(W)$ holds if and only if there exists a constant $C>0$ such that, for Lebesgue almost every $y \in (0,1)$, there exists $t_0(y) > 0$ such that for every $t \in (0, t_0(y))$, $ \mu(\Ball t y)\ge \exp(-C/t)$. 
Any constant $C > \esssup_{y\in(0,1)} A_\mu(y)$ can be used in this formulation. 
The radius threshold $t_0(y)$ need not be uniform, and an exceptional Lebesgue null set is allowed.

We record the finite compositions that are strict contractions and the positions at which they occur in an infinite word.

\begin{Def}\label{def:blocks}
Let $\{h_i\}_{i\in\A}$ be a family of maps from $[0,1]$ into itself with $\Lip(h_i)\le1$, and let $m\ge1$.
A word $\tau\in\A^m$ is called a \emph{strict block of length $m$} for this family if its composition $h_\tau$ is a strict contraction on $[0,1]$, that is, $\Lip(h_\tau)<1$.
Otherwise it is called \emph{nonstrict}.
Write
\[
 \begin{aligned}
 \mathcal S_m(h) &\coloneqq \left\{\tau\in\A^m \colon \Lip(h_\tau)<1 \right\},\\
 \mathcal U_m(h) &\coloneqq \A^m \setminus \mathcal S_m(h) = \{\tau\in\A^m \colon \Lip(h_\tau)=1\}.
 \end{aligned}
\]
For an infinite word $\omega\in\A^{\mathbb N}$, its length-$m$ block beginning at position $j\ge1$ is the word $\omega_j\cdots\omega_{j+m-1}$.
The set of starting positions of its strict blocks is
\[ \mathcal J_m(h;\omega) \coloneqq \left\{j\ge1 \colon \omega_j\cdots\omega_{j+m-1}\in\mathcal S_m(h) \right\}.\]
Occurrences may overlap, and repeated occurrences of the same word are counted at their respective starting positions.
We say that $\omega$ contains infinitely many strict blocks of length $m$ if $\mathcal J_m(h;\omega)$ is infinite.
Two occurrences are disjoint if their sets of symbol positions are disjoint.
For the source maps $\{f_i\}_i$, abbreviate $\mathcal S_m(f)$, $\mathcal U_m(f)$, and $\mathcal J_m(f;\omega)$ to $\mathcal S_m$, $\mathcal U_m$, and $\mathcal J_m(\omega)$, respectively.
The terms \emph{strict pair} and \emph{strict triple} refer to strict blocks of lengths $2$ and $3$, respectively.
\end{Def}

For $m\ge1$, define the set of infinite words with no strict length-$m$ block by
\[
 \begin{aligned}
 \Sigma_m^0 &\coloneqq \left\{\omega\in\A^{\mathbb N} \colon \mathcal J_m(\omega)=\emptyset \right\}\\
 &= \left\{ \omega \in \A^{\mathbb N} \colon \omega_j\cdots\omega_{j+m-1}\in\mathcal U_m \text{ for every } j \ge 1 \right\}.
 \end{aligned}
\]
The vanishing mesh of the source system $\{f_i\}_i$ defines a continuous coding map $\pi_f \colon \A^{\mathbb N}\to[0,1]$ by
\[ \{\pi_f(\omega)\} = \bigcap_{n\ge1} f_{\omega_1\cdots\omega_n}([0,1]).\]
For $x \in [0,1]\setminus D_f$, let $\omega(x) = \left(\omega_j(x) \right)_{j\ge1}$ be its unique \emph{address}, that is, the unique infinite word satisfying $\pi_f(\omega(x))=x$.
Uniqueness follows because $x$ lies in exactly one level-$n$ cylinder for each $n$.
For such $x$, define its level-$n$ cylinder by
\[ I_n(x)\coloneqq f_{\omega_1(x)\cdots\omega_n(x)}([0,1]), \qquad n\ge0.\]
We also define
\[ L_n(x)\coloneqq\Lip\left(f_{\omega_1(x)\cdots\omega_n(x)}\right), \qquad n\ge0.\]
The empty-word convention gives $I_0(x)=[0,1]$ and $L_0(x)=1$.
The sequence $(L_n(x))_{n\ge0}$ is positive and nonincreasing, since the source maps $\{f_i\}_i$ are strictly increasing and $1$-Lipschitz.

For each $m\ge1$, let  $K_m \coloneqq \pi_f \left(\Sigma_m^0 \right)$.
Thus $K_m$ is the set of points admitting an address that contains no strict block of length $m$.
The set $K_m$ is compact, since $\Sigma_m^0$ is a closed subset of the compact symbolic space and $\pi_f$ is continuous.

\begin{Def}\label{def:survivor}
We say that the source system $\{f_i\}_i$ satisfies the null survivor condition $(R)$ if there exists $m\ge1$ such that $\Leb(K_m)=0$.
\end{Def}

The survivor sets decrease as the block length increases.
Indeed, let $m'\ge m\ge1$.
Since every source map is $1$-Lipschitz, for $\tau\in\mathcal S_m$ and $\eta\in\A^{m'-m}$ we have
\[\Lip(f_{\tau\eta}) \le \Lip(f_\tau)\Lip(f_\eta)\le \Lip(f_\tau)<1.\]
Thus extending a strict block to the right produces another strict block, and hence
\[ \mathcal J_m(\omega)\subset\mathcal J_{m'}(\omega),\qquad \Sigma_{m'}^0\subset\Sigma_m^0,\qquad K_{m'}\subset K_m.\]
In particular, if $\Leb(K_m)=0$ for some $m$, then $\Leb(K_{m'})=0$ for every $m'\ge m$.
Therefore the block length used to verify $(R)$ may always be increased.

\begin{Thm}\label{thm:main}
Under the assumptions \eqref{eq:basic-f}--\eqref{eq:basic-g}, suppose that $(R)$ and $(W)$ hold.
Then $\mu$ satisfies condition $\lambda^*$ and, in particular, $\mu \in \Reg([0,1])$.
\end{Thm}

Neither the source maps $\{f_i\}_i$ nor the target maps $\{g_i\}_i$ are required to be differentiable.
No monotonicity of the source derivatives or prescribed convexity or concavity is assumed.
The proof uses both the conjugate equation and $(R)$.
Theorem~\ref{thm:main} does not assert that $(W)$ is sufficient for arbitrary measures.
In Section~\ref{sec:proof}, we construct the exceptional set determined by $(R)$.
Proposition~\ref{prop:pointwise} gives the pointwise estimate, including the stronger conclusion under $(UW)$.
The uniform lower estimate $(UW)$ was obtained for linear fractional systems in \cite[Proposition~2.13]{Okamura2026}.
The mesh condition $(M)$ introduced below is sufficient for this estimate and is not required in the theorem.

\section{Proof of the main result}\label{sec:proof}

We prove Theorem~\ref{thm:main} by excluding two Lebesgue null sets.
First, $(R)$ gives the exceptional set $Z_R$ in Lemma~\ref{lem:survivor}, outside which the prefix Lipschitz constants $L_n(x)$ tend to zero.
Second, $(W)$ gives a Lebesgue null set $E\subset(0,1)$ and a uniform bound $A_\mu(z)\le C$ for $z\in(0,1)\setminus E$.
We enlarge this set to $\bigcup_{\sigma\in\A^*}f_\sigma(E)$, which is null because every finite composition $f_\sigma$ is Lipschitz and $\A^*$ is countable.

For $x$ outside both exceptional sets, every inverse image $z=f_\sigma^{-1}(x)$ along its address belongs to $(0,1)\setminus E$.
In Proposition~\ref{prop:pointwise}, we use the conjugacy to transfer the bound at $z$ to $x$: letting the radius tend to zero with the prefix of length $k$ fixed gives $A_\mu(x)\le C L_k(x)$.
Letting $k\to\infty$ yields $A_\mu(x)=0$ for Lebesgue almost every $x$.
Finally, dominated convergence gives condition $\lambda^*$ and hence UST regularity.

The following lemma gives the exceptional null set associated with $(R)$ and proves that the prefix Lipschitz constants tend to zero outside this set.

\begin{Lem}\label{lem:survivor}
Assume \eqref{eq:basic-f} and $\Leb(K_m) = 0$. 
Let
\begin{equation*}
 Z_R \coloneqq D_f \cup \bigcup_{\sigma\in\A^*} f_\sigma(K_m).
\end{equation*}
Then $\Leb(Z_R)=0$. 
For every $x\notin Z_R$, we have
\begin{equation}\label{eq:Ldecay}
\lim_{n \to \infty} L_n (x) = 0.
\end{equation}
\end{Lem}

\begin{proof}
The set $D_f$ is countable, and every $f_\sigma$ is Lipschitz.
Thus each $f_\sigma(K_m)$ is null since $\Leb(K_m) = 0$.
Since $\A^*$ is countable, the union is null.

Let $x\notin Z_R$ and denote its unique address by $\omega$.
If $\mathcal J_m(\omega)$ were finite, we could write $\omega=\sigma\nu$ with $\nu \in\Sigma_m^0$.
Since $\pi_f (\nu) \in \pi_f \left( \Sigma_m^0 \right) = K_m$, we have $x\in f_\sigma(K_m)\subset Z_R$, a contradiction.
Hence $\mathcal J_m(\omega)$ is infinite.

Choose $j_1<j_2<\cdots$ in $\mathcal J_m(\omega)$ with $j_{k+1}\ge j_k + m$, so that the corresponding occurrences are pairwise disjoint.
Since $\mathcal S_m$ is finite and nonempty, $\kappa \coloneqq \max_{\tau\in\mathcal S_m} \Lip(f_\tau) < 1$.
For every $\ell \ge1$ and every $n\ge j_{\ell} +m-1$, the prefix $\omega_1\cdots\omega_n$ contains the first $\ell$ selected occurrences. 
By submultiplicativity,
\[ L_n(x) \le \prod_{s=1}^{\ell} \Lip\!\left(f_{\omega_{j_s}\cdots\omega_{j_s+m-1}}\right) \le \kappa^{\ell}, \qquad n\ge j_{\ell} +m-1.\]
Thus $\limsup_{n\to\infty} L_n(x)\le \kappa^{\ell}$ for every $\ell \ge1$.
Letting $\ell \to\infty$ and using $\kappa<1$, we obtain \eqref{eq:Ldecay}. 
\end{proof}

We combine \eqref{eq:Ldecay} with the mass conditions to obtain pointwise estimates for $\mu$.

\begin{Prop}\label{prop:pointwise}
Assume \eqref{eq:basic-f}--\eqref{eq:basic-g} and $\Leb(K_m) = 0$, and let $Z_R$ be the set in Lemma~\ref{lem:survivor}.
If $(W)$ holds, then
\begin{equation}\label{eq:pointwise}
 \lim_{r\downarrow0}r\log\mu(\Ball r x)=0
\end{equation}
for Lebesgue almost every $x\in[0,1]$.
If $(UW)$ holds, then \eqref{eq:pointwise} holds at every $x\notin Z_R$.
\end{Prop}

\begin{proof}
For every finite word $\sigma$, iterating the lower bound in \eqref{eq:basic-g} gives
\[
 g_\sigma(v)-g_\sigma(u)\ge c_g^{|\sigma|}(v-u),
 \qquad 0\le u<v\le1.
\]
Thus, for every interval $J=[u,v]\subset[0,1]$, the conjugacy \eqref{eq:conjugacy} gives
\begin{equation}\label{eq:mu-f-ineq}  
\mu(f_\sigma(J)) = g_\sigma(\varphi(v)) - g_\sigma(\varphi(u)) \ge c_g^{|\sigma|}\mu(J). 
\end{equation}
By $(W)$, choose $C<\infty$ and a Borel set $E\subset(0,1)$ such that $\Leb(E) = 0$ and $A_\mu(z)\le C$ for $z\in(0,1)\setminus E$.
Let
\begin{equation*}
 \mathcal N \coloneqq Z_R \cup \bigcup_{\sigma\in\A^*} f_\sigma(E).
\end{equation*}
Lemma~\ref{lem:survivor} gives $\Leb(Z_R)=0$.
Each $f_\sigma(E)$ is null because $f_\sigma$ is Lipschitz, so $\Leb(\mathcal N)=0$.

Fix $x \notin \mathcal N$ and a prefix $\sigma = \omega_1\cdots\omega_k$ of its address. 
Since $f_\sigma(0),f_\sigma(1)\in D_f$ and $x\notin D_f$, we have $z  = f_\sigma^{-1}(x) \in (0,1) \setminus E$. 
By the definition of $L_k(x)$, we have $\Lip(f_\sigma) = L_k(x)>0$.
For sufficiently small $r>0$, $\Ball{r/L_k(x)}z \subset (0,1)$ and $f_\sigma(\Ball{r/L_k(x)}z)\subset\Ball r x$.
Hence, by \eqref{eq:mu-f-ineq}, $\mu(\Ball r x)\ge c_g^k\mu(\Ball{r/L_k(x)}z)$.
Consequently,
\[ -r\log\mu(\Ball r x) \le kr|\log c_g|+L_k(x)\left(-\frac{r}{L_k(x)} \log\mu(\Ball{r/L_k(x)}z)\right).\]
First let $r\downarrow0$, keeping the prefix $\sigma$ fixed, to obtain
\begin{equation*}
 A_\mu(x)\le L_k(x) A_\mu(z)\le C L_k(x).
\end{equation*}
Then let $k\to\infty$ and apply \eqref{eq:Ldecay}.
This gives $A_\mu(x)=0$, which is equivalent to \eqref{eq:pointwise}.

Under $(UW)$, $A_\mu(z) \le C_1$ for every $z\in(0,1)$.
The same argument with $E = \emptyset$ proves the conclusion for every $x \notin Z_R$.
\end{proof}

\begin{proof}[Proof of Theorem~\ref{thm:main}]
By Proposition~\ref{prop:pointwise}, for each fixed $\eta>0$, the indicators of the sets in \eqref{eq:lambda} tend to $1$ Lebesgue almost everywhere as $s\to\infty$.
Dominated convergence gives \eqref{eq:lambda}.
Criterion $\lambda^*$ in \cite[Criterion 1.1]{ManticaTotik} then yields $\mu\in\Reg([0,1])$.
\end{proof}

\section{Finite equality sets}\label{sec:finite}

We give a sufficient condition for $(R)$ in Theorem \ref{thm:main} in terms of derivatives.
The resulting criterion also gives a countable exceptional set for the pointwise estimates under $(UW)$.

The condition requires the derivative to equal $1$ at only finitely many points for every composition of a fixed length.

\begin{Def}\label{def:finite}
Suppose that the source maps $\{f_i\}_i$ belong to $C^1([0,1])$.
We say that $(F)$ holds for $\{f_i\}_i$ if there exists an integer $p\ge1$ such that $E_\tau \coloneqq \{x\in[0,1]:f_\tau'(x)=1\}$ is finite for every $\tau \in \A^p$.
\end{Def}

The next lemma relates this finiteness condition to the occurrence of strict blocks along infinite words.

\begin{Lem}\label{lem:finite}
Let $\{h_i\}_{i\in\A}$ be compatible, with $h_i \in C^1([0,1])$, $0\le h_i'\le1$, and all sets $E_i = \{h_i'=1\}$ finite.
Then the mesh for $\{h_i\}_i$ vanishes, so the coding map $\pi_h$ is well-defined.
Let $A \coloneqq \bigcup_i h_i(E_i)$.
Then there exists an integer $m\ge1$, depending only on the family $\{h_i\}_i$, such that every infinite word $\omega$ satisfies the following assertions according to the cardinality of $\mathcal J_m(h;\omega)$:\\
(i) If $\#\mathcal J_m(h;\omega) = 0$, then $\pi_h(\omega)\in A$.\\
(ii) If $0<\#\mathcal J_m(h;\omega) < \infty$, then $\pi_h(\omega)\in\bigcup_{\sigma\in\A^*}h_\sigma(A)$.\\
(iii) If $\#\mathcal J_m(h;\omega) = \infty$, then $\lim_{n \to \infty} \Lip(h_{\omega_1\cdots\omega_n}) = 0$.
\end{Lem}

\begin{proof}
Since $E_i$ is finite, $h_i(y)-h_i(x) = \int_x^y h_i^{\prime}(t) dt <y-x$ for $x<y$.
For $t \in [0,1]$, define
\[ T(t) \coloneqq \max_{i\in\A}\max_{0\le x\le1-t}   \left(h_i(x+t)-h_i(x)\right).\]
This function is continuous and nondecreasing, with $T(0)=0$ and $T(t)<t$ for $t>0$.
Hence $\lim_{n \to \infty} T^n(1) = 0$ and in particular $\lim_{n \to \infty} \delta_n(h) = 0$.
Thus the intersection of the nested cylinders is a singleton and defines the coding map $\pi_h$.

We remark that the set $E \coloneqq \bigcup_i E_i$ is finite.
Choose $r \ge 1$ so large that every level-$r$ cylinder contains at most one point of $E$.
Such an $r$ exists because the mesh vanishes.
Let $m \coloneqq r+2$.

(i) First suppose that $\mathcal J_m(h;\omega) = \emptyset$, so every length-$m$ block of $\omega$ belongs to $\mathcal U_m(h)$. 
Since each such composition has Lipschitz constant $1$, its continuous derivative attains the value $1$ at some $u_j\in[0,1]$, that is, $(h_{\omega_j}\circ\cdots\circ h_{\omega_{j+m-1}})'(u_j)=1$.

For $j \ge 1$, let
\[ z_j \coloneqq h_{\omega_{j+1}}\circ\cdots\circ h_{\omega_{j+m-1}}(u_j), \qquad v_j \coloneqq h_{\omega_{j+2}}\circ\cdots\circ h_{\omega_{j+m-1}}(u_j). \]
Then
\[ (h_{\omega_{j+2}}\circ\cdots\circ h_{\omega_{j+m-1}})'(u_j) h_{\omega_{j+1}}' (v_j) h_{\omega_j}' (z_j)  = 1. \]
Hence $h_{\omega_j}' (z_j) = h_{\omega_{j+1}}' (v_j) = 1$, that is, $z_j \in E_{\omega_j}$ and $v_j \in E_{\omega_{j+1}}$.

Both $v_j$ and $z_{j+1}$ belong to $E \cap h_{\omega_{j+2}\cdots\omega_{j+m-1}}([0,1])$.
Thus $v_j = z_{j+1}$ and $z_j=h_{\omega_{j+1}}(z_{j+1})$.
It follows that $h_{\omega_1}(z_1)$ lies in every prefix cylinder of $\omega$, and hence $\pi_h(\omega)=h_{\omega_1}(z_1)\in A$, proving (i).

(ii) If $0<\#\mathcal J_m(h;\omega)<\infty$, then a tail of $\omega$ contains no strict length-$m$ block, so (i) applied to this tail gives (ii).

(iii) Finally, suppose that $\#\mathcal J_m(h;\omega)=\infty$.
Choose starting positions $j_k$ in this set with $j_{k+1}\ge j_k+m$.
These occurrences are pairwise disjoint, and
\[ \kappa=\max_{\tau\in\mathcal S_m(h)}\Lip(h_\tau)<1,\]
since $\mathcal S_m(h)$ is finite and nonempty.
If a prefix contains $k$ selected occurrences, submultiplicativity and $\Lip(h_i)\le1$ bound its Lipschitz constant by $\kappa^k$.
This tends to zero, proving (iii).
\end{proof}

Applying this lemma to finite compositions of the source maps $\{f_i\}_i$ gives regularity criteria and a countable exceptional set.

\begin{Prop}\label{prop:finite}
Assume \eqref{eq:basic-f}, with $f_i\in C^1([0,1])$ for every $i$. 
Fix an integer $p\ge1$ and assume that $E_\tau$ is finite for each $\tau\in\A^p$.
Let
\begin{equation}\label{eq:exceptional-countable}
 A_p \coloneqq \bigcup_{\tau\in\A^p}f_\tau(E_\tau),
 \qquad
Z_p \coloneqq D_f\cup\bigcup_{\sigma\in\A^*}f_\sigma(A_p).
\end{equation}
Then the following assertions hold.\\
(i) There exists an integer $m_0$ divisible by $p$ such that $K_{m_0} \subset A_p$, and $A_p$ is finite.
The set $Z_p$ is countable, and \eqref{eq:Ldecay} holds at every $x\notin Z_p$.\\
(ii) If the target maps $\{g_i\}_i$ satisfy \eqref{eq:basic-g} and $\mu$ satisfies $(W)$, then $\mu$ is UST regular and \eqref{eq:pointwise} holds Lebesgue almost everywhere.\\
(iii) If the target maps $\{g_i\}_i$ satisfy \eqref{eq:basic-g} and $\mu$ satisfies $(UW)$, then $\mu$ is UST regular and \eqref{eq:pointwise} holds at every $x\notin Z_p$.
\end{Prop}

\begin{proof}
Consider the alphabet $\mathcal B \coloneqq \A^p$ and the maps $h_\tau \coloneqq f_\tau$ for $\tau\in\mathcal B$.
Their image intervals are precisely the level-$p$ cylinders of the source system $\{f_i\}_i$, so the family $\{h_\tau\}_{\tau\in\mathcal B}$ is compatible when ordered lexicographically.
Each $h_\tau$ belongs to $C^1([0,1])$, and the chain rule gives $0 \le h_\tau' \le 1$.
Moreover, $\{x:h_\tau'(x)=1\}=E_\tau$ is finite by assumption.
Thus Lemma~\ref{lem:finite} applies, with $A$ equal to $A_p$.
Let $m$ be the integer supplied by the lemma, and let $m_0 \coloneqq pm$.
Here $m$ counts symbols in $\mathcal B$, whereas $m_0$ counts symbols in the original alphabet $\A$.

For any $\omega\in\A^{\mathbb N}$, group its symbols into consecutive words of length $p$ by setting
\[ \eta_j \coloneqq \left(\omega_{(j-1)p+1},\ldots,\omega_{jp} \right) \in\mathcal B, \quad j\ge1.\]
Then $\eta \coloneqq (\eta_j)_{j\ge1}\in\mathcal B^{\mathbb N}$ satisfies $h_{\eta_1\cdots\eta_n}=f_{\omega_1\cdots\omega_{np}}$ for every $n\ge 1$.
Consequently, $\pi_h(\eta)=\pi_f(\omega)$, since the cylinders for $\eta$ are the level-$np$ cylinders for $\omega$ and their diameters tend to zero.

We first prove the assertions in (i).
If $\omega\in\Sigma_{m_0}^0$, every length-$m_0$ block of $\omega$ is nonstrict.
For each $j\ge1$, the corresponding length-$m$ block of $\eta$ satisfies $h_{\eta_j\cdots\eta_{j+m-1}} = f_{\omega_{(j-1)p+1}\cdots\omega_{(j+m-1)p}}$.
Thus $\Lip(h_{\eta_j\cdots\eta_{j+m-1}})=1$, and $\mathcal J_m(h;\eta)=\emptyset$.
Lemma~\ref{lem:finite}(i) gives $\pi_f(\omega)=\pi_h(\eta)\in A_p$.
Therefore $K_{m_0} = \pi_f \left(\Sigma_{m_0}^0\right) \subset A_p$.
The set $A_p$ is finite because it is a finite union of images of the finite sets $E_\tau$.
Since $D_f$ and $\A^*$ are countable, the definition of $Z_p$ in \eqref{eq:exceptional-countable} shows that $Z_p$ is countable.

Next fix $x\notin Z_p$.
Since $D_f\subset Z_p$, the point $x$ has a unique address $\omega$, and $L_n(x)$ is well-defined for each $n \ge 1$.
Form the corresponding grouped address $\eta$ as above, so that $x=\pi_h(\eta)$.
If $\mathcal J_m(h;\eta)$ were finite, Lemma~\ref{lem:finite}(i) or (ii) and the definition of $Z_p$ in \eqref{eq:exceptional-countable} would give
\[ x\in\bigcup_{\beta\in\mathcal B^*} h_\beta(A_p) \subset\bigcup_{\sigma\in\A^*}f_\sigma(A_p) \subset Z_p,\]
which is a contradiction.
Here we use $A_p = h_{\emptyset}(A_p)$ and $\emptyset \in \mathcal B^*$.

Therefore $\mathcal J_m(h;\eta)$ is infinite, and Lemma~\ref{lem:finite}(iii) yields
\[ \lim_{n \to \infty} L_{np}(x) = \lim_{n \to \infty} \Lip(h_{\eta_1\cdots\eta_n}) = 0. \]
This gives convergence first along the lengths divisible by $p$.
Since each $f_i$ is $1$-Lipschitz, the sequence $L_n(x)$ is nonincreasing in $n$.
Thus, for an arbitrary integer $k\ge p$, we have $ 0\le L_k(x)\le L_{p\lfloor k/p\rfloor}(x)$.
The right-hand side tends to zero as $k\to\infty$, proving \eqref{eq:Ldecay} and completing (i).

For (ii) and (iii), assume also that the target maps $\{g_i\}_i$ satisfy \eqref{eq:basic-g}.
Since $K_{m_0}\subset A_p$ and $A_p$ is finite, $\Leb(K_{m_0})=0$.
For this choice of $m_0$, the exceptional set in Lemma~\ref{lem:survivor} satisfies
\[ Z_R = D_f\cup\bigcup_{\sigma\in\A^*}f_\sigma(K_{m_0}) \subset D_f\cup\bigcup_{\sigma\in\A^*}f_\sigma(A_p) = Z_p.\]
If $(W)$ holds, Proposition~\ref{prop:pointwise} gives \eqref{eq:pointwise} Lebesgue almost everywhere, and Theorem~\ref{thm:main} gives UST regularity, proving (ii).
If $(UW)$ holds, Proposition~\ref{prop:pointwise} gives \eqref{eq:pointwise} at every $x\notin Z_R$, hence at every $x\notin Z_p$.
Since $(UW)$ implies $(W)$, UST regularity follows from (ii), proving (iii).
\end{proof}

\section{Uniform mesh bounds}\label{sec:mesh}

We give mass estimates and regularity criteria under a uniform bound on the source mesh.
We also give consequences of $C^2$ regularity and contraction tests for finite compositions.

We impose a uniform $O(1/n)$ bound on the lengths of source cylinders $f_{\sigma}([0,1])$.

\begin{Def}\label{def:mesh}
The source system $\{f_i\}_i$ satisfies the uniform mesh bound condition $(M)$ if there exists $C_0 > 0$ such that $\delta_n(f)\le\frac{C_0}{n+1}$ for every $n \ge 0$.
\end{Def}

This is the bound in \cite[Lemma~8.4, equation~(8.10)]{ManticaTotik}.

Under \eqref{eq:basic-f} and $(M)$, the source maps $\{f_i\}_i$ satisfy the following estimate.
Choose $C_0>0$ as in $(M)$.
For $x\in[0,1]\setminus D_f$ and integers $0\le k<n$, splitting the address of $x$ after its first $k$ symbols gives
\[ |I_n(x)|\le L_k(x)\delta_{n-k}(f) \le L_k(x)\frac{C_0}{n-k+1}.\]
Multiplying by $n$ and taking the upper limit as $n\to\infty$ with $k$ fixed, we obtain
\[ 0\le\limsup_{n\to\infty}n|I_n(x)|\le C_0L_k(x).\]
Consequently, at every $x\in[0,1]\setminus D_f$ for which $\lim_{k \to \infty} L_k (x) = 0$, we have
\begin{equation}\label{eq:littleo}
 \lim_{n\to\infty}n|I_n(x)|=0.
\end{equation}

The uniform mesh bound gives $(UW)$ and, together with $(R)$, yields UST regularity and the cylinder estimates below.

\begin{Prop}\label{prop:mesh}
Assume \eqref{eq:basic-f}--\eqref{eq:basic-g} and $(M)$.
Then the following hold.\\
(i) The measure $\mu$ satisfies $(UW)$.\\
(ii) If $(R)$ holds, then $\mu$ is UST regular, and furthermore, both \eqref{eq:pointwise} and \eqref{eq:littleo} hold at every $x\notin Z_R$, where $Z_R$ is the Lebesgue null set in Lemma~\ref{lem:survivor}.\\
(iii) If the source maps $\{f_i\}_i$ belong to $C^1([0,1])$ and satisfy $(F)$, then $\mu$ is UST regular.
Assume that for some $p \ge 1$, $E_\tau \coloneqq \{x\in[0,1]:f_\tau'(x)=1\}$ is finite for every $\tau \in \A^p$, and let $Z_p$ be the countable set in \eqref{eq:exceptional-countable}.
Then both \eqref{eq:pointwise} and \eqref{eq:littleo} hold at every $x\notin Z_p$.
\end{Prop}

\begin{proof}
For (i), choose $C_0>0$ as in $(M)$.
The conjugacy \eqref{eq:conjugacy} and iteration of the lower bound in \eqref{eq:basic-g} give
\begin{equation*}
 \mu(f_\sigma([0,1])) = g_\sigma(1)-g_\sigma(0) \ge c_g^{|\sigma|}.
\end{equation*}
Fix $y\in[0,1]$ and $0<t\le1$, and take $n \coloneqq \lceil C_0/t\rceil$.

Assume that $y \in f_{\sigma}([0,1])$ and $|\sigma| = n$.  
By $(M)$, $f_{\sigma}(1) - f_{\sigma}(0) \le \frac{C_0}{n+1} < t$ and hence $f_{\sigma}([0,1]) \subset \Ball{t}{y}$. 
Now we see that 
\[ \mu(\Ball t y)\ge \mu\left( f_{\sigma}([0,1]) \right) \ge c_g^n \ge \exp\!\left(-\frac{(C_0+1)|\log c_g|}{t}\right).\]
This proves $(UW)$ and hence $(W)$, completing (i).

For (ii), assume $(R)$.
By (i), Proposition~\ref{prop:pointwise} gives \eqref{eq:pointwise} at every $x\notin Z_R$, and Theorem~\ref{thm:main} gives UST regularity.
For the cylinder estimate, Lemma~\ref{lem:survivor} gives $\lim_{k \to \infty} L_k (x) = 0$ for every $x\notin Z_R$.
Therefore the estimate under $(M)$ following Definition~\ref{def:mesh}  gives \eqref{eq:littleo} at every such point.

For (iii), assume that the source maps $\{f_i\}_i$ belong to $C^1([0,1])$ and satisfy $(F)$, and choose $p$ and $Z_p$ as in the statement.
By (i), Proposition~\ref{prop:finite}(iii) gives UST regularity of $\mu$ and \eqref{eq:pointwise} at every $x\notin Z_p$.
Proposition~\ref{prop:finite}(i) gives $\lim_{k \to \infty} L_k(x)=0$ at every such point, so the estimate for $(M)$ following Definition~\ref{def:mesh} also gives \eqref{eq:littleo} there.
\end{proof}

We next derive $(F)$ and $(R)$ from $(M)$ for $C^2$ source maps $\{f_i\}_i$.

\begin{Prop}\label{prop:C2mesh}
Assume \eqref{eq:basic-f} and $(M)$, and suppose that $f_i \in C^2([0,1])$ for each $i$.
Then $(F)$ and $(R)$ hold.
\end{Prop}

\begin{proof}
We first prove that there exists an integer $n_0\ge1$ such that, for every $n\ge n_0,\ \tau\in\A^n$,
\begin{equation}\label{eq:C2endpoint-equality}
 \{x\in[0,1] \colon f_\tau'(x)=1\}\subset\{0,1\}. 
\end{equation}

Suppose that $\tau=(\tau_1,\ldots,\tau_n)\in\A^n$ and $f_\tau'(z)=1$ for some $z\in(0,1)$.
Let $i_j=\tau_{n-j}$ for $0 \le j < n$, and define
\[ z_0 \coloneqq z,\qquad z_{j+1} \coloneqq f_{i_j}(z_j), \quad 0 \le j < n.\]
Thus $z_j=f_{\tau_{n-j+1}\cdots\tau_n}(z)$ for $1\le j\le n$, and $z_n = f_\tau(z)$.
Since each $f_i$ is strictly increasing, induction gives $z_j \in (0,1)$ for every $j$.
By the chain rule, $1=f_\tau'(z)=\prod_{j=0}^{n-1} f_{i_j}'(z_j)$.
Since $f_{i}' \in [0,1]$ for each $i$, $f_{i_j}'(z_j)=1$ for every $0 \le j < n$.
Since the differentiable function $f_{i_j}'$ attains its maximum at the interior point $z_j$, we have $f_{i_j}''(z_j)=0$ for $0 \le j < n$.

Fix $0<t\le1/2$.
At least one of $z-t$ and $z+t$ belongs to $[0,1]$, and choose such a point and denote it by $z'$.
Define the corresponding images of $z'$ and their distances from the images of $z$ by
\[
 \begin{aligned}
 y_0& \coloneqq z',\qquad y_{j+1} \coloneqq f_{i_j}(y_j), \quad 0\le j<n,\\
 d_j&\coloneqq|y_j-z_j|, \qquad 0\le j\le n.
 \end{aligned}
\]
In particular, $d_0=t$ and $d_n=|f_\tau(z')-f_\tau(z)|$.
Then $0<d_{j+1}\le d_j\le t, \ 0\le j<n$.

For $0\le j<n$, let $u_j \coloneqq y_j-z_j$, so that $|u_j|=d_j$.
Taylor's formula with integral remainder, applied at $z_j$ with increment $u_j$, gives
\[
 \begin{aligned}
 y_{j+1}-z_{j+1}
 &=f_{i_j}(z_j+u_j)-f_{i_j}(z_j)\\
 &=f_{i_j}'(z_j)u_j  +u_j^2\int_0^1(1-s)f_{i_j}''(z_j+s u_j)\,ds\\
 &=u_j+u_j^2\int_0^1(1-s)      \left(f_{i_j}''(z_j+s u_j)-f_{i_j}''(z_j)\right)\,ds.
 \end{aligned}
\]
The last equality uses $f_{i_j}'(z_j)=1$ and $f_{i_j}''(z_j)=0$.
The formula holds for either sign of $u_j$.

Define a common modulus of continuity of the second derivatives by
\[ \omega(t)\coloneqq\max_{i\in\A}  \sup_{\substack{u,v\in[0,1]\\|u-v|\le t}} |f_i''(u)-f_i''(v)|, \qquad t \in [0,1]. \]

The segment joining $z_j$ and $y_j$ is contained in $[0,1]$, and $|s u_j|\le d_j\le t$ for $0\le s\le1$.
Now the definition of $\omega$ gives that 
\[ \left|(y_{j+1}-z_{j+1})-u_j\right| \le u_j^2\int_0^1(1-s)\omega(t)\,ds =\frac12\omega(t)d_j^2.\]
By the reverse triangle inequality,
\[ d_{j+1}\ge d_j-\frac12\omega(t)d_j^2 \ge d_j-\frac12\omega(t)t^2.\]
Summing over $j=0,\ldots,n-1$ and using $d_0=t$, we obtain
\begin{equation}\label{eq:C2neutral-lower}
 \delta_n(f) \ge d_n = |f_\tau(z')-f_\tau(z)| \ge t-\frac n2\omega(t)t^2.
\end{equation}
The first inequality follows because both $z$ and $z'$ belong to $[0,1]$.

Choose $C_0>0$ as in $(M)$ and fix $L>C_0$. 
Since the family is finite and the second derivatives are continuous on $[0,1]$, we have $\lim_{t\downarrow0} \omega(t) = 0$.
Hence there exists $n_0\ge1$ such that, for every $n\ge n_0$, $\frac{L}{n} \le\frac{1}{2}$ and $\omega\left(\frac{L}{n}\right)<\frac{2(L-C_0)}{L^2}$.
For these $n$, taking $t=L/n$ in \eqref{eq:C2neutral-lower} gives
\[ n\delta_n(f)\ge L-\frac{L^2}{2} \omega\left(\frac{L}{n}\right)>C_0, \]
which contradicts $(M)$. 
The choice of $n_0$ is independent of $\tau$ and $z$, so \eqref{eq:C2endpoint-equality} follows.

Fix $p\ge n_0$.
By \eqref{eq:C2endpoint-equality}, every set $E_\tau$, $\tau\in\A^p$, is finite, so $(F)$ holds.
Proposition~\ref{prop:finite}(i) gives $K_{m_0}\subset A_p$ for some $m_0$, with $A_p$ finite, and hence $(R)$ holds.
\end{proof}

We record the resulting pointwise estimates outside the set of cylinder endpoints.

\begin{Cor}\label{cor:C2mesh}
Assume \eqref{eq:basic-f} and $(M)$, and suppose that every $f_i$ belongs to $C^2([0,1])$.
Then the following hold.\\
(i) For every $x\in[0,1]\setminus D_f$, we have $\lim_{n\to\infty}L_n(x)=0$ and \eqref{eq:littleo}.\\
(ii) If the target maps $\{g_i\}_i$ also satisfy \eqref{eq:basic-g}, then $\mu$ satisfies $(UW)$ and is UST regular.
Moreover, \eqref{eq:pointwise} holds at every $x\in[0,1]\setminus D_f$.
\end{Cor}

\begin{proof}
For (i), choose $p\ge n_0$ as in the proof of Proposition~\ref{prop:C2mesh}.
Equation~\eqref{eq:C2endpoint-equality} gives $A_p\subset D_f$, and hence $Z_p=D_f$, because $D_f$ is invariant under all finite compositions.
Therefore Proposition~\ref{prop:finite}(i) gives $\lim_{k \to \infty} L_k (x) = 0$ for every $x\notin D_f$.
The estimate for $(M)$ following Definition~\ref{def:mesh} then gives \eqref{eq:littleo} at every such point.

For (ii), assume that the target maps $\{g_i\}_i$ satisfy \eqref{eq:basic-g}.
Proposition~\ref{prop:mesh}(i) gives $(UW)$.
Proposition~\ref{prop:finite}(iii), with $Z_p=D_f$, shows that $\mu$ is UST regular and that \eqref{eq:pointwise} holds at every $x\notin D_f$.
\end{proof}

\subsection{A modulus for finite compositions}

A common quadratic contraction modulus for finite compositions gives the mesh bound $(M)$.

\begin{Def}\label{def:quadratic}
We say that $(Q)$ holds if there exist an integer $q\ge1$ and $a>0$ such that
\[   f_\tau(y)-f_\tau(x)  \le \frac{y-x}{1+a(y-x)},  \]
for every $0\le x<y\le1,\ \tau\in\A^q$.
\end{Def}

Iteration of the modulus in $(Q)$ gives an explicit mesh bound and hence $(M)$.

\begin{Cor}\label{cor:block}
Assume \eqref{eq:basic-f} and $(Q)$.
Then $(M)$ holds.
\end{Cor}

\begin{proof}
Choose $q$ and $a$ as in Definition~\ref{def:quadratic}.
Split a word into $\lfloor n/q\rfloor$ consecutive blocks of length $q$ and a remainder.
The remainder is nonexpanding, while the common modulus $T(t)=t/(1+at)$ is increasing and satisfies $T^k(1)=1/(1+ak)$.
Thus, for every integer $n\ge0$,
\begin{equation*}\label{eq:block-mesh}
 \delta_n(f)\le\frac{1}{1+a\lfloor n/q\rfloor}.
\end{equation*}
In particular, $(M)$ holds.
\end{proof}

Under \eqref{eq:basic-g}, Proposition~\ref{prop:mesh} gives $(UW)$.
If $(R)$ also holds, the UST regularity and pointwise conclusions of that proposition apply, with exceptional set $Z_R$.
For $C^1$ source maps $\{f_i\}_i$, $(F)$ may replace $(R)$.
Proposition~\ref{prop:finite} allows the exceptional set to be taken as the countable set $Z_p$.

We give three sufficient conditions for the quadratic contraction estimate in $(Q)$.

\begin{Prop}\label{prop:check}
Let $h\in C^1([0,1])$ be strictly increasing and satisfy $0\le h'\le1$.
Each of the following conditions (i)--(iii) implies that
\[ h(y)-h(x)\le \frac{y-x}{1+a(y-x)}, \quad x < y, \]
for some $a>0$. \\
(i) $\Lip(h)<1$.\\
(ii) $E=\{h'=1\}$ is nonempty and finite, and for some $b>0$,
\[  1-h'(x)\ge b\,\dist(x,E), \qquad 0\le x\le1. \] 
(iii) There exists $b>0$ such that for every interval $J \subset [0,1]$,
\begin{equation*}
 \int_J(1-h'(u))\,du\ge b|J|^2.
\end{equation*}
\end{Prop}

\begin{proof}
For (i), if $\Lip(h)=\lambda<1$, use $\lambda t\le t/(1+(\lambda^{-1}-1)t)$ for $0\le t\le1$.
For (ii), let $M \coloneqq \#(E) \ge 1$.
Let $J$ be an interval in $[0,1]$ and let $t \coloneqq |J| \in [0,1]$.
Since
\[ \Leb\left(\{u\in J \colon \dist(u,E)\le s\}\right) \le 2Ms, \]
Tonelli's theorem gives
\[ \int_J\dist(u,E)\,du = \int_J \left(\int_0^{1} {\mathbf 1}_{\{\dist(u,E) > s\}} ds \right) du  \ge\int_0^{t/(2M)}(t-2Ms)\,ds = \frac{t^2}{4M}. \]
Thus (ii) implies (iii), with constant $b/(4M)$.
Under (iii),
\[ h(y)-h(x)\le t-bt^2\le\frac{t}{1+bt},\qquad t=y-x. \qedhere\]
\end{proof}

\section{The assumptions of Mantica--Totik}\label{sec:MT}

We show that the assumptions of Mantica--Totik imply the uniform mesh bound $(M)$ and the null survivor condition $(R)$.
Their invariant-measure theorem then follows from our criterion.

We recall assumptions A-0--A-4 from \cite[Section~8]{ManticaTotik}, using the notation of this paper.

\begin{Def}\label{def:MT}
A family of maps $f_i:[0,1]\to[0,1]$, $i\in\A$, satisfies the \emph{Mantica--Totik assumptions} if the following conditions hold.
\begin{enumerate}
\item[\textup{A-0.}] The maps are strictly increasing weak contractions, that is,
\[ 0<f_i(y)-f_i(x)<y-x,  \qquad 0\le x<y\le1,\ i\in\A.\]
The image intervals $[f_i(0),f_i(1)]$ cover $[0,1]$ and have pairwise disjoint interiors.
The maps are indexed from left to right, so that $f_0(0)=0$, $f_i(1)=f_{i+1}(0)$ for $0\le i<N-1$, and $f_{N-1}(1)=1$.
Let $\xi_i\in[f_i(0),f_i(1)]$ be the unique fixed point of $f_i$; in particular, $\xi_0=0$ and $\xi_{N-1}=1$.

\item[\textup{A-1.}] The map $f_0$ is concave on $[0,1]$, and there exists $\rho_0>0$ such that
\[ \frac{f_0(x)}{x}\le1-\rho_0 x,  \qquad 0<x\le1.\]

\item[\textup{A-2.}] The map $f_{N-1}$ is convex on $[0,1]$, and there exists $\rho_{N-1}>0$ such that
\[ \frac{1-f_{N-1}(x)}{1-x}\le1-\rho_{N-1}(1-x),  \qquad 0\le x<1.\]

\item[\textup{A-3.}] For every $1\le i\le N-2$, the fixed point satisfies $\xi_i\in(f_i(0),f_i(1))$.
The map $f_i$ is concave on $[\xi_i,1]$ and convex on $[0,\xi_i]$, and there exist $\rho_i,\rho_i^*>0$ such that
\[ \frac{f_i(x)-\xi_i}{x-\xi_i}\le1-\rho_i(x-\xi_i),  \qquad \xi_i<x\le1,\]
and
\[ \frac{\xi_i-f_i(x)}{\xi_i-x}\le1-\rho_i^*(\xi_i-x),  \qquad 0\le x<\xi_i.\]

\item[\textup{A-4.}] All positive constants introduced in A-1--A-3 are less than $1$.
Thus all the inequalities remain valid when these constants are replaced by the common value
\[ \rho\coloneqq\min\left( \{\rho_0,\rho_{N-1}\}\cup \{\rho_i,\rho_i^*:1\le i\le N-2\} \right) \in (0,1).\]
\end{enumerate}
\end{Def}

The existence and uniqueness of each fixed point $\xi_i$ follow from A-0.
On the compact interval $[0,1]$, the weak contraction condition in A-0 is equivalent to the weak contraction condition used in  \cite[Subsection~2.2]{Okamura2026}.

Recall that
\[ K_2 = \pi_f(\Sigma_2^0),\quad \Sigma_2^0 = \left\{\omega\in\A^{\mathbb N}\colon \Lip(f_{\omega_j}\circ f_{\omega_{j+1}})=1 \text{ for every }j\ge1\right\}.\]
Thus $K_2$ consists of the points admitting an address with no strict pair.
The following proposition gives a mesh bound and shows that $K_2$ is contained in the set of points fixed by at least one of the maps.

\begin{Prop}\label{prop:MT}
Let $\{f_i\}_{i\in\A}$ satisfy A-0--A-4 in Definition~\ref{def:MT}.
Then the following hold.\\
\textup{(i)} Condition $(M)$ holds.\\
\textup{(ii)} For the fixed points $\xi_i$ from Definition~\ref{def:MT}, we have
\begin{equation}\label{eq:MT-survivor}
 K_2\subset\{\xi_0,\ldots,\xi_{N-1}\}.
\end{equation}
In particular, condition $(R)$ holds with $m=2$.\\
\textup{(iii)}  For every target system $\{g_i\}_i$ satisfying \eqref{eq:basic-g}, the conjugacy measure $\mu$ satisfies $(UW)$ and is UST regular.\\
\textup{(iv)} Under the same assumption as in \textup{(iii)}, \eqref{eq:pointwise} holds outside the countable set
\[  D_f\cup\bigcup_{\sigma\in\A^*} f_\sigma\left(\{\xi_0,\ldots,\xi_{N-1}\}\right). \] 
\end{Prop}

\begin{proof}
For (i), the mesh estimate $(M)$ is given by \cite[Lemma~8.4, equation~(8.10)]{ManticaTotik}. 

By A-0, the source maps $\{f_i\}_i$ are strictly increasing weak contractions, and their image intervals form a compatible partition.
Since $(M)$ implies vanishing mesh, the source maps $\{f_i\}_i$ satisfy \eqref{eq:basic-f}, and the coding map $\pi_f$ and the sets $K_m$ are defined.

For (ii), let $\xi_i,\rho$ be as in Definition~\ref{def:MT}.
We have $\xi_0=0$, $\xi_{N-1}=1$, and $f_i(0)<\xi_i<f_i(1)$ for $0<i<N-1$.
By A-1--A-4, the common constant $\rho\in(0,1)$ satisfies
\begin{equation*}
 0<\frac{f_i(x)-\xi_i}{x-\xi_i}
 \le1-\rho|x-\xi_i|, \qquad x\ne \xi_i.
\end{equation*}
For an endpoint fixed point, we use only the inequality on the nonempty side.

We claim that there exists $\kappa<1$ such that $ \Lip(f_i\circ f_j)\le \kappa$ for every $i\ne j$.

If $j>i$, then $f_j([0,1])\subset[f_i(1),1]$ and $f_i(1)>\xi_i$.
For $f_i(1)\le u<v\le1$, concavity of $f_i$ on $[\xi_i,1]$ gives
\[ \frac{f_i(v)-f_i(u)}{v-u} \le \frac{f_i (v) - \xi_i}{v-\xi_i} \le \frac{f_i(f_i(1))-\xi_i}{f_i(1)-\xi_i} \le 1-\rho(f_i(1)-\xi_i) < 1.\]
Hence $ \Lip\left(f_i|_{[f_i(1), 1]}\right) < 1$. 
Since $\Lip(f_j)\le1$, it follows that $\Lip(f_i\circ f_j)$ is strictly less than $1$.

If $j<i$, then $f_j([0,1])\subset[0,f_i(0)]$ and $f_i(0) < \xi_i$.
For $0 \le u < v \le f_i (0)$, convexity of $f_i$ on $[0,\xi_i]$ similarly gives
\[ \Lip\left(f_i|_{[0,f_i(0)]}\right)  \le\frac{\xi_i-f_i(f_i(0))}{\xi_i-f_i(0)} \le1-\rho(\xi_i-f_i(0))<1.\]
There are finitely many pairs, so a common $\kappa<1$ exists.

Consequently, for every $\omega\in\Sigma_2^0$ and $r\ge1$, $\Lip(f_{\omega_r}\circ f_{\omega_{r+1}})=1$ implies that $\omega_r=\omega_{r+1}$. 
Thus every such $\omega$ is constant, and $ \Sigma_2^0\subset\{i^\infty \colon i\in\A\}$.
For each $i\in\A$, $\xi_i \in f_i^n([0,1])$, and the vanishing mesh implies
\[ \{\pi_f(i^\infty)\} =\bigcap_{n\ge1}f_i^n([0,1]) = \{\xi_i\}.\]
Therefore, 
\[ K_2=\pi_f(\Sigma_2^0)  \subset\{\pi_f(i^\infty):i\in\A\} =\{\xi_0,\ldots,\xi_{N-1}\},\]
which is \eqref{eq:MT-survivor}.
In particular, $\Leb(K_2)\le\sum_{i\in\A}\Leb(\{\xi_i\})=0$, so $(R)$ holds with $m=2$.

For (iii), assume that the target maps $\{g_i\}_i$ satisfy \eqref{eq:basic-g}.
By (i), Proposition~\ref{prop:mesh}(i) gives $(UW)$, and hence $(W)$.
Together with $(R)$ from (ii), Theorem~\ref{thm:main} gives UST regularity of $\mu$.

For (iv), the exceptional set from Lemma~\ref{lem:survivor}, with $m=2$, satisfies
\[ Z_R=D_f\cup\bigcup_{\sigma\in\A^*}f_\sigma(K_2) \subset D_f\cup\bigcup_{\sigma\in\A^*} f_\sigma\left(\{\xi_0,\ldots,\xi_{N-1}\}\right).\]
The set on the right is countable.
Under $(UW)$, Proposition~\ref{prop:pointwise} gives \eqref{eq:pointwise} outside this set.
\end{proof}

Choosing affine target maps $\{g_i\}_i$ recovers the balanced measures considered by Mantica--Totik.

\begin{Rem}\label{rem:MT-balanced}
For weights $p_i>0$ with $\sum_i p_i=1$, choose the affine target maps $\{g_i\}_i$ given by
\[ g_i(y)=\sum_{j<i}p_j+p_i y.\]
They satisfy \eqref{eq:basic-g}, and the conjugacy measure is the balanced measure of the source IFS $\{f_i\}_i$.
Indeed, $\mu(f_i(J))=p_i\mu(J)$ for every interval $J$, and hence for every Borel set by uniqueness of finite measures.
The source images partition $[0,1]$ up to their null endpoints.
Thus Proposition~\ref{prop:MT} recovers \cite[Theorem~8.1]{ManticaTotik} and also allows nonlinear target maps $\{g_i\}_i$.

The secant argument does not require $C^1$ regularity.
At every point $x\ne \xi_i$ where the derivative exists, it gives $ f_i'(x)\le1-\rho|x-\xi_i|<1$.
Thus, if the maps are $C^1$, then $\{f_i'=1\}\subset\{\xi_i\}$, so $(F)$ holds with $p=1$.
Without this regularity assumption, the set of differentiability points where $f_i'=1$ is still contained in $\{\xi_i\}$.
However, A-0--A-4 do not assume differentiability at every point or continuity of the derivative.
Condition $(R)$ does not require differentiability.
\end{Rem}

\section{Endpoint criteria and linear fractional systems}\label{sec:endpoints}

We give sufficient conditions in terms of the behavior of the source maps $\{f_i\}_i$ at the endpoints.
We first show that the uniform mesh bound implies quadratic endpoint drift without differentiability assumptions.
We then consider this drift condition for $C^1$ maps and give differential criteria for LF systems.

Under the next two conditions, the source maps $\{f_i\}_i$ can have derivative $1$ only at the endpoints.

\begin{Def}\label{def:endpoint}
For $C^1$ source maps $\{f_i\}_i$, let $E_i \coloneqq \{x\in[0,1]:f_i'(x)=1\}$.
(i) We say that condition $(H_\partial)$ holds if $E_i \subset\{0,1\}$ for every $i \in \A$. \\
(ii) We say that condition $(H)$ holds if $E_i \subset\{0,1\}$ and $E_i \ne \{0,1\}$ for every $i \in \A$. 
\end{Def}

For a compatible family of $C^1$, $1$-Lipschitz maps, condition $(H_\partial)$ implies vanishing mesh by Lemma~\ref{lem:finite}, since every $E_i$ is finite.
Thus the vanishing mesh requirement in \eqref{eq:basic-f} is automatic under these assumptions.

We impose quadratic lower bounds on the displacement under the two maps that fix endpoints.

\begin{Def}\label{def:drift}
The source system $\{f_i\}_i$ satisfies the drift condition $(D)$ if
\[  \liminf_{x\downarrow0}\frac{x-f_0(x)}{x^2}>0 
 \quad \textup{ and } \quad
 \liminf_{x\uparrow1}\frac{f_{N-1}(x)-x}{(1-x)^2}>0.\] 
\end{Def}

Quadratic endpoint drift is necessary for the uniform mesh bound, even without differentiability of the source maps $\{f_i\}_i$.

\begin{Lem}\label{lem:mesh-drift}
Under \eqref{eq:basic-f}, condition $(M)$ implies $(D)$.
Consequently, $(Q)$ also implies $(D)$.
\end{Lem}

\begin{proof}
Assume $(M)$, and choose $C_0>0$ such that $\delta_n(f)\le C_0/(n+1)$ for every $n\ge0$.
Let $d(x)\coloneqq x-f_0(x)$.
Since $f_0(0)=0$ and $f_0$ is increasing and $1$-Lipschitz, we have $0\le f_0(x)\le x$.
Moreover, for $0\le x<y\le1$,
$d(y)-d(x)=(y-x)-(f_0(y)-f_0(x))\ge0$,
so $d$ is nondecreasing.
For every $n\ge1$ and $x\in[0,1]$, telescoping gives
\[ x-f_0^n(x) =\sum_{j=0}^{n-1}d(f_0^j(x)) \le n\,d(x),\]
because $f_0^j(x)\le x$.
On the other hand, $(M)$ gives
\[ 0\le f_0^n(x)\le f_0^n(1) =f_0^n(1)-f_0^n(0) \le\delta_n(f)\le\frac{C_0}{n+1}.\]
Fix $x\in(0,1]$ and choose $n=\lceil 2C_0/x\rceil$.
Then $f_0^n(x)\le x/2$, and hence $\frac{x}{2}\le x-f_0^n(x)\le n\,d(x)$.
Since $n\le2C_0/x+1$, it follows that
\[ \frac{x-f_0(x)}{x^2} \ge\frac{1}{2nx} \ge\frac{1}{4C_0+2x}.\]
Therefore,
\[ \liminf_{x\downarrow0}\frac{x-f_0(x)}{x^2} \ge\frac{1}{4C_0}>0.\]

For the right endpoint, define $h(t)\coloneqq1-f_{N-1}(1-t)$.
This map is increasing and $1$-Lipschitz, satisfies $h(0)=0$, and has the bound
\[h^n(1)=1-f_{N-1}^n(0) \le\delta_n(f)\le\frac{C_0}{n+1}.\]
The same argument applied to $h$ gives
\[ \liminf_{x\uparrow1}\frac{f_{N-1}(x)-x}{(1-x)^2} =\liminf_{t\downarrow0}\frac{t-h(t)}{t^2} \ge\frac{1}{4C_0}>0.\]
Thus $(D)$ holds.
Finally, Corollary~\ref{cor:block} gives $(Q)\Rightarrow(M)$, and hence $(Q)\Rightarrow(D)$.
\end{proof}

We next estimate the length of the image of $[0,1]$ under repeated application of a single source map.

\begin{Lem}\label{lem:run-bound}
Assume \eqref{eq:basic-f}, with $f_i\in C^1([0,1])$ for every $i\in\A$, and suppose that $(H_\partial)$ and $(D)$ hold.
There exists a constant $C>0$ such that for every $i\in\A,\ r\ge1$, 
\begin{equation}\label{eq:run-bound}
\left|f_i^r([0,1])\right| \le\frac{C}{r+1}.
\end{equation}
\end{Lem}

\begin{proof}
Define
\[ \mathcal P \coloneqq \left\{i\in\A \colon   (i=0\text{ and }f_0'(0)=1)   \text{ or }(i=N-1\text{ and }f_{N-1}'(1)=1) \right\}.\]

For $i \notin \mathcal P$, we claim that $\Lip(f_i^2)<1$.
Otherwise, continuity of the derivative and the chain rule give a point $z\in[0,1]$ such that $f_i'(f_i(z))=f_i'(z)=1$.
By condition $(H_\partial)$,  $z, f_i(z)\in\{0,1\}$.
By compatibility and strict monotonicity of $\{f_i\}_i$, either $(i,z)=(0,0)$ or $(i,z)=(N-1,1)$. 
Hence $i\in\mathcal P$, a contradiction.
Since $|f_i^r([0,1])| \le  \Lip(f_i^2)^{\lfloor r/2\rfloor}, r \ge 1$, \eqref{eq:run-bound} holds for such $i$. 

For $i=0\in\mathcal P$, let $u_r \coloneqq f_0^r(1)$.
Since $f_0(0)=0$, we have
\[ 0 < u_r=f_0^r(1)-f_0^r(0)  \le \delta_r(f) \rightarrow 0, \qquad r\to\infty.\]
The convergence follows from the vanishing mesh assumption in \eqref{eq:basic-f}.
By condition $(D)$, there exists $a > 0$ such that $u_{r+1}\le u_r-au_r^2$ for all sufficiently large $r$. 
Hence
\[ \frac{1}{u_{r+1}} - \frac{1}{u_r}  =\frac{u_r-u_{r+1}}{u_ru_{r+1}} \ge a,\]
which proves \eqref{eq:run-bound}. 

For $i=N-1\in\mathcal P$, apply the same argument to $v_r\coloneqq 1-f_{N-1}^r(0)$, using the second condition in $(D)$.
Taking the maximum over the finitely many indices $i\in\A$ gives a common constant $C$ in \eqref{eq:run-bound}.
\end{proof}

We combine the preceding estimate with contraction at changes of symbols to obtain a uniform mesh bound and decay of prefix Lipschitz constants.

\begin{Prop}\label{prop:endpoint}
Assume \eqref{eq:basic-f}, with $f_i\in C^1([0,1])$ for every $i\in\A$, and suppose that $(H_\partial)$ and $(D)$ hold.
Then the following hold.\\
(i) Condition $(M)$ holds.\\
(ii) For every $x\in[0,1]\setminus D_f$, we have $\lim_{k\to\infty}L_k(x)=0$, and \eqref{eq:littleo} holds.\\
(iii) If the target maps $\{g_i\}_i$ also satisfy \eqref{eq:basic-g}, then $\mu$ satisfies $(UW)$ and is UST regular.
Moreover, \eqref{eq:pointwise} holds at every $x\notin D_f$.
\end{Prop}

\begin{proof}
For (i), we first estimate the Lipschitz constants of threefold compositions.
Suppose that $\Lip(f_i\circ f_j\circ f_k)=1$.
Continuity of the derivative gives a point $z\in[0,1]$ at which $(f_i\circ f_j\circ f_k)'(z)=1$.
Since all three derivative factors lie in $[0,1]$, the chain rule gives
\[ f_i'(f_j(f_k(z)))=f_j'(f_k(z))=f_k'(z)=1.\]
By $(H_\partial)$,  $z,\ f_k(z),\ f_j(f_k(z))\in\{0,1\}$. 
By compatibility and strict monotonicity, $f_k(z)\in\{0,1\}$ forces either $(k,z)=(0,0)$ or $(k,z)=(N-1,1)$.
In the first case, $f_j(f_k(z))=f_j(0)\in\{0,1\}$ forces $j=0$, and in the second case, $f_j(f_k(z))=f_j(1)\in\{0,1\}$ forces $j=N-1$.
Thus $j=k$.
Since there are only finitely many triples, there exists $\rho\in(0,1)$ such that
\[ \Lip(f_i\circ f_j\circ f_k) \le \rho \qquad i,j,k\in\A,\  j\ne k.\]

Write an arbitrary nonempty word as a product of maximal constant runs: 
\[ \sigma=a_1^{r_1}\cdots a_K^{r_K},\quad a_j\ne a_{j+1},\quad r_j\ge1,\quad n=|\sigma|=\sum_{j=1}^K r_j. \]

For $1\le j\le K$, let $P_j \coloneqq a_1^{r_1}\cdots a_{j-1}^{r_{j-1}}$.
Write $\sigma=\sigma_1\cdots\sigma_n$ and let
\[ b_0 \coloneqq 0,\qquad b_s \coloneqq \sum_{\ell=1}^s r_\ell, \qquad 1\le s\le K.\]
Thus the boundary between runs $s$ and $s+1$ lies between positions $b_s$ and $b_s+1$.

For each $j$, define
\[ \mathcal C_j \coloneqq \{s\in\mathbb Z \colon 2\le s\le j-2,\ s\equiv2\pmod3\}.\]
Then $\#\mathcal C_j=\left\lfloor\frac{j-1}{3}\right\rfloor$. 
For $s\in\mathcal C_j$, define
\[ \beta_s \coloneqq \sigma_{b_s-1}\sigma_{b_s}\sigma_{b_s+1}
 =
 \begin{cases}
 a_s a_s a_{s+1},&r_s\ge2,\\
 a_{s-1}a_s a_{s+1},&r_s=1.
 \end{cases}
\]
Since $b_s-1\ge1$ and $b_s+1\le b_{j-1}=|P_j|$, this occurrence of $\beta_s$ lies entirely within $P_j$.
Its last two symbols are distinct, so the preceding bound for threefold compositions gives $\Lip(f_{\beta_s})\le\rho$.

For consecutive indices $s,s+3\in\mathcal C_j$, we have
\[ b_{s+3}-b_s=r_{s+1}+r_{s+2}+r_{s+3} \ge 3.\]
Hence the sets of positions $\{b_s-1,b_s,b_s+1\}, \ s\in\mathcal C_j$, are pairwise disjoint.
Using $\Lip(f_i)\le1$ for the remaining symbols, we obtain
\[ \Lip(f_{P_j}) \le\prod_{s\in\mathcal C_j}\Lip(f_{\beta_s}) \le\rho^{\#\mathcal C_j} =\rho^{\lfloor(j-1)/3\rfloor},\]
where an empty product is interpreted as $1$.

Let $L\coloneqq f_\sigma(1)-f_\sigma(0)$.
For $1\le j\le K$, let $Q_j\coloneqq a_{j+1}^{r_{j+1}}\cdots a_K^{r_K}$, the suffix after the corresponding run.
Here $P_1 = Q_K = \emptyset$.
Then $\sigma$ has the decomposition $ \sigma= P_j a_j^{r_j} Q_j$. 
Let $C$ be the constant in Lemma~\ref{lem:run-bound}.
The suffix satisfies $f_{Q_j}([0,1])\subset[0,1]$, the prefix satisfies the preceding Lipschitz bound, and the $j$th run satisfies \eqref{eq:run-bound}.
Combining these estimates gives
\[ L \le \left|f_{P_j}\left(f_{a_j}^{r_j}([0,1])\right)\right| \le\rho^{\lfloor(j-1)/3\rfloor}\,  |f_{a_j}^{r_j}([0,1])| \le\frac{C\rho^{\lfloor(j-1)/3\rfloor}}{r_j+1}.\]

Consequently,
\[  nL=\sum_{j=1}^K r_jL \le C\sum_{j=1}^K\rho^{\lfloor(j-1)/3\rfloor} \le\frac{3C}{1-\rho}.\]
Enlarging the constant gives $(M)$ and proves (i).

For (ii), condition $(H_\partial)$ implies $(F)$ with $p=1$.
In the notation of Proposition~\ref{prop:finite}, we have $A_1=\bigcup_{i\in\A}f_i(E_i)\subset D_f$, since $f_i(E_i) \subset \{f_i (0), f_i (1)\}$. 
Since $D_f$ is invariant under all finite compositions, $Z_1=D_f$.
Proposition~\ref{prop:finite}(i) now gives $\lim_{k \to \infty} L_k(x) = 0$ for every $x\notin D_f$.
By (i) and the consequence of $(M)$ established after Definition~\ref{def:mesh}, \eqref{eq:littleo} holds at every such point.
This proves (ii).

For (iii), assume that the target maps $\{g_i\}_i$ satisfy \eqref{eq:basic-g}.
By (i), Proposition~\ref{prop:mesh}(i) gives $(UW)$.
Proposition~\ref{prop:finite}(iii), with $p=1$ and $Z_1=D_f$, then shows that $\mu$ is UST regular and that \eqref{eq:pointwise} holds at every $x\notin D_f$.
\end{proof}

The first condition in $(D)$ holds automatically if $f_0'(0)<1$, and the second holds automatically if $f_{N-1}'(1)<1$.
Thus only neutral fixed endpoints require an additional condition.
No analogous condition is needed at nonfixed endpoints of other maps.

\subsection{Differential conditions and LF systems}

We next impose lower bounds on $1-f_0'(x)$ near $0$ and on $1-f_{N-1}'(x)$ near $1$.

\begin{Def}\label{def:differential}
For $C^1$ source maps $\{f_i\}_i$, the differential endpoint conditions are
\begin{equation}\label{eq:differential-endpoints}
 \liminf_{x\downarrow0}\frac{1-f_0'(x)}{x}>0
 \quad\textup{and}\quad
 \liminf_{x\uparrow1}\frac{1-f_{N-1}'(x)}{1-x}>0.
\end{equation}
\end{Def}

Combined with $(H_\partial)$, these differential conditions imply the drift condition and a quadratic modulus for twofold compositions.

\begin{Prop}\label{prop:ND1}
Assume $f_i \in C^1([0,1])$ for each $i$, \eqref{eq:basic-f} and $(H_\partial)$.
Suppose that \eqref{eq:differential-endpoints} holds.
Then $(D)$ and $(Q)$ hold, with $q=2$ in $(Q)$.
\end{Prop}

\begin{proof}
We first verify $(D)$.
By \eqref{eq:differential-endpoints}, there exist constants $c_0,c_1>0$ and $\varepsilon\in(0,1)$ such that
\[ 1-f_0'(t)\ge c_0t, \quad 0<t<\varepsilon, \qquad 1-f_{N-1}'(t)\ge c_1(1-t), \quad 1-\varepsilon<t<1.\]
Since $f_0(0)=0$, integration gives
\[ x-f_0(x) =\int_0^x \left(1-f_0'(t)\right)\,dt \ge c_0\int_0^x t\,dt =\frac{c_0}{2}x^2,  \qquad 0<x<\varepsilon.\]
Similarly, since $f_{N-1}(1)=1$, we obtain that 
\[ f_{N-1}(x)-x  =\int_x^1 \left(1-f_{N-1}'(t)\right)\,dt \ge c_1\int_x^1 (1-t)\,dt = \frac{c_1}{2}(1-x)^2, \qquad 1-\varepsilon<x<1.\]
Dividing by $x^2$ and $(1-x)^2$, respectively, and taking the lower limits at the corresponding endpoints, we obtain $(D)$.

We next prove $(Q)$ with $q=2$. 
For $h=f_i\circ f_j$, equality $h'(x)=1$ forces $f_j(x)\in E_i\subset\{0,1\}$ and $f_j'(x)=1$.
Compatibility forces $(j,x)=(0,0)$ or $(j,x)=(N-1,1)$.
Thus $\{h'=1\}$ is empty or consists of a single endpoint. 

If $\{h'=1\} = \emptyset$, then $\Lip(h)<1$.
If $\{h'=1\} = \{0\}$, then $j=0$, $f_0'(0)=1$, and
\[ 1-h'(x)=1-f_i'(f_0(x))f_0'(x)\ge1-f_0'(x).\] 
The first inequality in \eqref{eq:differential-endpoints} gives a positive linear lower bound near $0$.
Since $\{h'=1\}=\{0\}$, continuity gives $\min_{[\varepsilon,1]}(1-h')>0$.
After decreasing the constant, the linear lower bound holds on all of $[0,1]$, so Proposition~\ref{prop:check}(ii) applies. 
The case of $\{h'=1\} = \{1\}$ follows from the second inequality in \eqref{eq:differential-endpoints}.

For each pair $(i,j)\in\A^2$, Proposition~\ref{prop:check} gives a constant $a_{ij}>0$ such that
\[ (f_i\circ f_j)(y)-(f_i\circ f_j)(x)
 \le \frac{y-x}{1+a_{ij}(y-x)}, 
 \qquad 0\le x<y\le1.\]
Since there are only finitely many pairs, $a\coloneqq\min_{i,j\in\A}a_{ij} > 0$. 
Since replacing $a_{ij}$ by $a$ increases the right-hand side, $(Q)$ holds with $q=2$.
\end{proof}

For $C^2$ source maps $\{f_i\}_i$, nondegeneracy at neutral fixed endpoints can be expressed in terms of second derivatives.

\begin{Def}\label{def:ND}
We say that $C^2$ source maps $\{f_i\}_i$ satisfy condition $(ND)$ if both of the following hold:
\begin{equation*}\label{eq:ND}
 \begin{aligned}
 f_0'(0)<1&\quad\textup{or}\quad f_0''(0)<0,\\
 f_{N-1}'(1)<1&\quad\textup{or}\quad f_{N-1}''(1)>0.
 \end{aligned}
\end{equation*}
\end{Def}

These conditions give a $C^2$ regularity criterion that includes LF source systems.

\begin{Cor}\label{cor:C2LF}
Assume \eqref{eq:basic-f}--\eqref{eq:basic-g}, with $f_i\in C^2([0,1])$ for each $i$, and suppose that $(H_\partial)$ and $(ND)$ hold.
Then $\mu$ is UST regular.
\end{Cor}

\begin{proof}
We first verify the two inequalities in \eqref{eq:differential-endpoints}.
Since $f_0\in C^2([0,1])$, the Taylor expansion of $f_0'$ at $0$ gives
\[ f_0'(x)=f_0'(0)+f_0''(0)x+o(x) \qquad x\downarrow0.\]
Subtracting from $1$ and dividing by $x>0$, we obtain
\[ \frac{1-f_0'(x)}x =\frac{1-f_0'(0)}x-f_0''(0)+o(1), \qquad x\downarrow0. \]

If $f_0'(0)<1$, the quotient tends to $+\infty$ as $x\downarrow0$.

If $f_0'(0)=1$, the first term vanishes, and $(ND)$ gives $f_0''(0)<0$.
Therefore,
\[ \lim_{x\downarrow0}\frac{1-f_0'(x)}x = -f_0''(0) > 0.\]

At the right endpoint, the Taylor expansion of $f_{N-1}'$ at $1$ gives
\[  f_{N-1}'(x) =f_{N-1}'(1)-f_{N-1}''(1)(1-x)+o(1-x), \qquad x\uparrow1.\]
Thus, after subtracting from $1$ and dividing by $1-x>0$, we have
\[ \frac{1-f_{N-1}'(x)}{1-x} = \frac{1-f_{N-1}'(1)}{1-x}+f_{N-1}''(1)+o(1), \qquad x\uparrow1.\]

If $f_{N-1}'(1)<1$, this quotient tends to $+\infty$ as $x\uparrow1$.

If $f_{N-1}'(1)=1$, condition $(ND)$ gives $f_{N-1}''(1)>0$, and hence
\[ \lim_{x\uparrow1}\frac{1-f_{N-1}'(x)}{1-x} =f_{N-1}''(1)>0.\]
Since the endpoint derivatives are at most $1$, these cases prove \eqref{eq:differential-endpoints}.
Proposition~\ref{prop:ND1} gives $(D)$, so Proposition~\ref{prop:endpoint}(iii) yields UST regularity of $\mu$.
\end{proof}

The same source assumptions give the following contraction and mass estimates.

\begin{Cor}\label{cor:C2-estimates}
Assume \eqref{eq:basic-f}, with $f_i\in C^2([0,1])$ for every $i\in\A$, and suppose that $(H_\partial)$ and $(ND)$ hold.
Then the following hold.\\
(i) Conditions $(D)$, $(Q)$, and $(M)$ hold, with $q=2$ in $(Q)$.\\
(ii) For every $x\in[0,1]\setminus D_f$, we have $\lim_{k\to\infty}L_k(x)=0$, and \eqref{eq:littleo} holds.\\
(iii) If the target maps $\{g_i\}_i$ also satisfy \eqref{eq:basic-g}, then $\mu$ satisfies $(UW)$, and \eqref{eq:pointwise} holds at every $x\notin D_f$. 
\end{Cor}

\begin{proof}
For (i), the Taylor expansion in the proof of Corollary~\ref{cor:C2LF} gives \eqref{eq:differential-endpoints} using only the source assumptions.
Therefore Proposition~\ref{prop:ND1} gives $(D)$ and $(Q)$ with $q=2$, and Proposition~\ref{prop:endpoint}(i) gives $(M)$.
Part (ii) follows from Proposition~\ref{prop:endpoint}(ii).
Under the additional target assumption, Proposition~\ref{prop:endpoint}(iii) gives (iii).
\end{proof}

We recall the definition of an LF system from \cite[Definition~2.8]{Okamura2026}, using the compatibility condition in Definition~\ref{def:compatible}.

\begin{Def}\label{def:LF}
A compatible family $\{h_i\}_{i\in\A}$ on $[0,1]$ is called a \emph{linear fractional system}, or an \emph{LF system}, if there are real coefficients $a_i,b_i,c_i,d_i$ such that, for every $i\in\A$,
\[h_i(x)=\frac{a_ix+b_i}{c_ix+d_i}, \quad 0\le x\le1, \qquad d_i>0,\quad c_i+d_i>0,\]
and
\[ 0<\Delta_i\coloneqq a_id_i-b_ic_i \le\min\{d_i^2,(c_i+d_i)^2\}.\]
\end{Def}

The regularity criterion applies to LF source systems and, when both systems are LF, also to the inverse conjugacy.

\begin{Cor}\label{cor:LF}
Suppose that the source maps $\{f_i\}_i$ form an LF system in the sense of Definition~\ref{def:LF}, and that the target maps $\{g_i\}_i$ satisfy \eqref{eq:basic-g}.
Then the following hold.\\
(i) The measure $\mu=\mu_\varphi$ is UST regular.\\
(ii) If the target maps $\{g_i\}_i$ also form an LF system, then both $\mu_\varphi$ and $\mu_{\varphi^{-1}}$ are UST regular.
\end{Cor}

\begin{proof}
We verify that every LF source system satisfies the source assumptions of Corollary~\ref{cor:C2LF}.
With the coefficients and $\Delta_i$ from Definition~\ref{def:LF}, we have
\[ f_i'(x)=\frac{\Delta_i}{(c_ix+d_i)^2}, \qquad f_i''(x)=-\frac{2c_i\Delta_i}{(c_ix+d_i)^3}.\]
If $c_i\ne0$, the derivative is strictly monotone and $E_i$ is empty or consists of a single endpoint.
If $c_i = 0$, then the map is affine with slope strictly less than $1$, because its image has length less than $1$.
If $f_0'(0)=1$, then $c_0>0$: $c_0<0$ would give a derivative greater than $1$ inside the interval, and $c_0=0$ would make $f_0$ the identity.
Thus $f_0''(0)<0$.
Similarly, $f_{N-1}'(1)=1$ implies $c_{N-1}<0$ and $f_{N-1}''(1)>0$.
This proves that every LF source system satisfies $(H)$ and $(ND)$.
Its maps are $C^\infty$, so Corollary~\ref{cor:C2LF} proves (i).
If the target maps $\{g_i\}_i$ also form an LF system, both families satisfy the source and target assumptions.
Applying (i) with the two systems interchanged proves UST regularity of $\mu_{\varphi^{-1}}$, and hence (ii).
\end{proof}

Let $\mathfrak m\colon[0,1]\to[0,1]$ denote Minkowski's question mark function.
Corollary~\ref{cor:LF}(ii) applies to the measures with distribution functions $\mathfrak m$ and $\mathfrak m^{-1}$.
Indeed, take $N=2$ and
\[ f_0(x)=\frac{x}{1+x},\qquad f_1(x)=\frac{1}{2-x},\qquad g_0(x)=\frac{x}{2},\qquad g_1(x)=\frac{1+x}{2}.\]
The functional identities
\[ \mathfrak m\!\left(\frac{x}{1+x}\right) =\frac{\mathfrak m(x)}{2},\qquad \mathfrak m\!\left(\frac{1}{2-x}\right) =\frac{1+\mathfrak m(x)}{2}\]
identify the conjugacy $\varphi$ with $\mathfrak m$.
Both the source maps $\{f_i\}_i$ and the target maps $\{g_i\}_i$ form LF systems, and interchanging the two systems gives the inverse conjugacy $\mathfrak m^{-1}$.
Hence Corollary~\ref{cor:LF}(ii) yields UST regularity of both $\mu_{\mathfrak m}$ and $\mu_{\mathfrak m^{-1}}$, recovering \cite[Theorems~1.2 and~1.3]{ManticaTotik}.

For LF source systems, the known mesh estimate provides an alternative proof of the regularity conclusions.

\begin{Rem}\label{rem:LF-alternative}
Suppose that the source maps $\{f_i\}_i$ form an LF system and the target maps $\{g_i\}_i$ satisfy \eqref{eq:basic-g}.
The LF mesh bound in \cite[Lemma~2.12]{Okamura2026} gives $(M)$.
Since every LF source map is $C^\infty$, Corollary~\ref{cor:C2mesh} gives $(UW)$, UST regularity of $\mu$, and \eqref{eq:pointwise} and \eqref{eq:littleo} at every $x\notin D_f$.
If both systems are LF, applying this argument with the two systems interchanged also gives UST regularity of $\mu_{\varphi^{-1}}$.
Thus, when both systems are LF, this strengthens the conclusion of \cite[Proposition~2.13]{Okamura2026} to UST regularity for both measures.
\end{Rem}

\section{Implications and differentiability assumptions}\label{sec:implications}

We summarize the implications between the conditions and the differentiability assumptions required for each.
All implications below hold under the standing assumptions \eqref{eq:basic-f}--\eqref{eq:basic-g}.
In particular, the target maps $\{g_i\}_i$ have Lipschitz inverses on their image intervals and need not be differentiable.
The differentiability assumptions in the following subsections concern only the source maps $\{f_i\}_i$.
In the diagrams, solid arrows denote implications, and dotted lines without arrowheads join the two assumptions to the conjunction $(W)\land(R)$.

\begingroup
\tikzset{
 ust condition/.style={draw,rounded corners=1.5pt,minimum height=7mm,
  inner xsep=4pt,inner ysep=3pt,align=center,font=\small},
 ust compound/.style={ust condition,fill=black!4},
 ust implication/.style={-{Latex[length=1.8mm,width=1.2mm]},line width=.55pt},
 ust equivalent/.style={{Latex[length=1.8mm,width=1.2mm]}-{Latex[length=1.8mm,width=1.2mm]},line width=.55pt},
 ust input/.style={densely dotted,line width=.7pt}
}

\subsection{\texorpdfstring{$C^0$}{C0} source maps}\label{subsec:implications-C0}

We impose no differentiability assumption on the source maps $\{f_i\}_i$; they are $1$-Lipschitz by \eqref{eq:basic-f}.
The conditions $(W)$, $(UW)$, $(R)$, $(M)$, $(Q)$, $(D)$, and $\lambda^*$ are defined without assuming differentiability of the source maps $\{f_i\}_i$.
Corollary~\ref{cor:block}, Lemma~\ref{lem:mesh-drift}, Proposition~\ref{prop:mesh}, and Theorem~\ref{thm:main} give the implications shown in the following diagram.
Proposition~\ref{prop:MT} gives both $(M)$ and $(R)$ under the Mantica--Totik assumptions A-0--A-4.

\begin{center}
\begin{tikzpicture}[x=1cm,y=1cm]
 \node[ust condition] (q) at (0,0) {$(Q)$};
 \node[ust condition] (m) at (2.3,0) {$(M)$};
 \node[ust condition] (d) at (2.3,-1.6) {$(D)$};
 \node[ust condition] (uw) at (4.6,0) {$(UW)$};
 \node[ust condition] (w) at (6.9,0) {$(W)$};
 \node[ust condition] (r) at (6.9,2) {$(R)$};
 \node[ust compound] (mt) at (2.3,2) {Mantica--Totik\\A-0--A-4};
 \node[ust compound] (both) at (9.2,1) {$(W)\land(R)$};
 \node[ust condition] (lambda) at (9.2,-.7) {$\lambda^*$};
 \node[ust condition] (ust) at (9.2,-2.2) {UST regularity};
 \draw[ust implication] (q) -- (m);
 \draw[ust implication] (m) -- (d);
 \draw[ust implication] (m) -- (uw);
 \draw[ust implication] (uw) -- (w);
 \draw[ust implication] (mt) -- (m);
 \draw[ust implication] (mt) -- (r);
 \draw[ust input] (r.east) -- (both.north west);
 \draw[ust input] (w.east) -- (both.south west);
 \draw[ust implication] (both) -- (lambda);
 \draw[ust implication] (lambda) -- (ust);
\end{tikzpicture}
\end{center}

The mesh condition $(M)$ implies $(W)$ through $(UW)$, whereas the main criterion also requires $(R)$.
Lemma~\ref{lem:mesh-drift} also gives $(M)\Rightarrow(D)$ without differentiability assumptions.
We discuss the reverse implication under $(H_\partial)$ in the next subsection.

\subsection{\texorpdfstring{$C^1$}{C1} source maps}\label{subsec:implications-C1}

For $C^1$ source maps $\{f_i\}_i$, Proposition~\ref{prop:finite} gives the upper chain below.
The lower chain and the regularity criterion are as in the preceding subsection.

\begin{center}
\begin{tikzpicture}[x=1cm,y=1cm]
 \node[ust condition] (h) at (0,2.2) {$(H)$};
 \node[ust condition] (hp) at (2.3,2.2) {$(H_\partial)$};
 \node[ust condition] (f) at (4.6,2.2) {$(F)$};
 \node[ust condition] (r) at (6.9,2.2) {$(R)$};
 \node[ust condition] (q) at (0,0) {$(Q)$};
 \node[ust condition] (m) at (2.3,0) {$(M)$};
 \node[ust condition] (d) at (2.3,-1.6) {$(D)$};
 \node[ust condition] (uw) at (4.6,0) {$(UW)$};
 \node[ust condition] (w) at (6.9,0) {$(W)$};
 \node[ust compound] (both) at (9.2,1.1) {$(W)\land(R)$};
 \node[ust condition] (lambda) at (9.2,-.6) {$\lambda^*$};
 \node[ust condition] (ust) at (9.2,-2.1) {UST regularity};
 \draw[ust implication] (h) -- (hp);
 \draw[ust implication] (hp) -- (f);
 \draw[ust implication] (f) -- (r);
 \draw[ust implication] (q) -- (m);
 \draw[ust implication] (m) -- (d);
 \draw[ust implication] (m) -- (uw);
 \draw[ust implication] (uw) -- (w);
 \draw[ust input] (r.east) -- (both.north west);
 \draw[ust input] (w.east) -- (both.south west);
 \draw[ust implication] (both) -- (lambda);
 \draw[ust implication] (lambda) -- (ust);
\end{tikzpicture}
\end{center}

Proposition~\ref{prop:endpoint} gives $(M)$ under $(H_\partial)$ and $(D)$.
Together with Lemma~\ref{lem:mesh-drift}, this implies that $(D)$ and $(M)$ are equivalent under $(H_\partial)$.
Here $(D)$ is a condition on function values, whereas $(H_\partial)$ requires $C^1$ regularity.
The differential endpoint conditions \eqref{eq:differential-endpoints} in Definition~\ref{def:differential} imply $(D)$ and, together with $(H_\partial)$, imply $(Q)$ with $q=2$ by Proposition~\ref{prop:ND1}.
In the following diagram, \emph{derivative bounds} denotes \eqref{eq:differential-endpoints}.

\begin{center}
\begin{tikzpicture}[x=1cm,y=1cm]
 \node[ust compound] (strong) at (0,0) {$(H_\partial)$ and\\derivative bounds};
 \node[ust condition] (bounds) at (3,1.6) {Derivative bounds};
 \node[ust condition] (d) at (7.3,1.6) {$(D)$};
 \node[ust compound] (hd) at (7.3,0) {$(H_\partial)\land(D)$};
 \node[ust condition] (q) at (3,-1.6) {$(Q)$ with $q=2$};
 \node[ust condition] (m) at (10,0) {$(M)$};
 \draw[ust implication] (strong) -- (bounds);
 \draw[ust implication] (bounds) -- (d);
 \draw[ust implication] (strong) -- (hd);
 \draw[ust implication] (hd) -- (d);
 \draw[ust implication] (hd) -- (m);
 \draw[ust implication] (m.north) -- (d.east);
 \draw[ust implication] (strong) -- (q);
 \draw[ust implication] (q) -- (m.south west);
\end{tikzpicture}
\end{center}

\subsection{\texorpdfstring{$C^2$}{C2} source maps}\label{subsec:implications-C2}

For $C^2$ source maps $\{f_i\}_i$, Proposition~\ref{prop:C2mesh} gives the additional implication $(M)\Rightarrow(F)$.
Thus $(M)$ implies both $(W)$ and $(R)$ and is sufficient for UST regularity in this setting.
Proposition~\ref{prop:ND1}, together with the Taylor expansion in the proof of Corollary~\ref{cor:C2LF}, gives $(Q)$ with $q=2$ under $(H_\partial)$ and $(ND)$.
An LF source system in the sense of Definition~\ref{def:LF} has $C^\infty$ maps and satisfies $(H)$ and $(ND)$.
We obtain the following diagram.

\begin{center}
\begin{tikzpicture}[x=1cm,y=1cm]
 \node[ust compound] (lf) at (0,4.8) {LF source system};
 \node[ust compound] (hnd) at (0,3.2) {$(H)\land(ND)$};
 \node[ust compound] (hpnd) at (0,1.6) {$(H_\partial)\land(ND)$};
 \node[ust condition] (q) at (0,0) {$(Q)$ with $q=2$};
 \node[ust condition] (m) at (2.4,0) {$(M)$};
 \node[ust condition] (f) at (4.8,1.2) {$(F)$};
 \node[ust condition] (r) at (7.2,1.2) {$(R)$};
 \node[ust condition] (uw) at (4.8,-1.2) {$(UW)$};
 \node[ust condition] (w) at (7.2,-1.2) {$(W)$};
 \node[ust compound] (both) at (9.6,0) {$(W)\land(R)$};
 \node[ust condition] (lambda) at (9.6,-1.7) {$\lambda^*$};
 \node[ust condition] (ust) at (9.6,-3.2) {UST regularity};
 \draw[ust implication] (lf) -- (hnd);
 \draw[ust implication] (hnd) -- (hpnd);
 \draw[ust implication] (hpnd) -- (q);
 \draw[ust implication] (q) -- (m);
 \draw[ust implication] (m) -- (f);
 \draw[ust implication] (f) -- (r);
 \draw[ust implication] (m) -- (uw);
 \draw[ust implication] (uw) -- (w);
 \draw[ust input] (r.east) -- (both.north west);
 \draw[ust input] (w.east) -- (both.south west);
 \draw[ust implication] (both) -- (lambda);
 \draw[ust implication] (lambda) -- (ust);
\end{tikzpicture}
\end{center}

Remark~\ref{rem:LF-alternative} gives an alternative argument for LF systems: the known mesh bound gives $(M)$ directly, and the two implications from $(M)$ in the diagram then apply.

For $C^2$ source maps, the endpoint conditions satisfy the following equivalences without assuming $(H_\partial)$.
The implication from $(M)$ to $(D)$ follows from Lemma~\ref{lem:mesh-drift}.
The term \emph{derivative bounds} again denotes \eqref{eq:differential-endpoints}.

\begin{center}
\begin{tikzpicture}[x=1cm,y=1cm]
 \node[ust condition] (m) at (0,0) {$(M)$};
 \node[ust condition] (d) at (2.3,0) {$(D)$};
 \node[ust condition] (nd) at (4.6,0) {$(ND)$};
 \node[ust condition] (bounds) at (8.1,0) {Derivative bounds};
 \draw[ust implication] (m) -- (d);
 \draw[ust equivalent] (d) -- (nd);
 \draw[ust equivalent] (nd) -- (bounds);
\end{tikzpicture}
\end{center}

Indeed, at a neutral left endpoint, Taylor's formula gives $(x-f_0(x))/x^2\to-f_0''(0)/2$, and the corresponding ratio at a neutral right endpoint tends to $f_{N-1}''(1)/2$.
At a fixed endpoint with derivative less than $1$, the corresponding ratio tends to $+\infty$.
The equivalence of the differential endpoint conditions with $(ND)$ follows from the Taylor expansions in the proof of Corollary~\ref{cor:C2LF}.
\endgroup

\subsection{Scope of the criterion}\label{subsec:scope}

Theorem~\ref{thm:main} imposes no differentiability assumption on either the source maps $\{f_i\}_i$ or the target maps $\{g_i\}_i$.
No convexity or concavity assumption is imposed on the source maps $\{f_i\}_i$.
Proposition~\ref{prop:MT} shows that assumptions A-0--A-4 of \cite{ManticaTotik} imply both $(M)$ and $(R)$, so \cite[Theorem~8.1]{ManticaTotik} is recovered for affine target maps $\{g_i\}_i$.
Corollary~\ref{cor:LF} strengthens the uniform lower mass estimate of \cite[Proposition~2.13]{Okamura2026} to UST regularity.

The mesh condition $(M)$ gives the uniform mass estimate $(UW)$ by Proposition~\ref{prop:mesh}.
The main criterion requires only $(W)$ and $(R)$, which do not in general imply $(UW)$, as shown in Example~\ref{ex:W-not-UW}.
Thus $(M)$ gives a uniform estimate stronger than the mass condition in Theorem~\ref{thm:main}.

\section{Examples and limitations}\label{sec:examples}

We give examples illustrating the roles of longer strict blocks and interior neutral fixed points and distinguishing the geometric and mass conditions above.
The final example shows that $(W)$ is not necessary for UST regularity among general measures.

Among the explicit source systems in this section, only Example~\ref{ex:interior} satisfies all the Mantica--Totik assumptions A-0--A-4.
Each of the other source systems violates at least one of these assumptions, although its conjugacy measure is UST regular for every target system $\{g_i\}_i$ satisfying \eqref{eq:basic-g}.

\subsection{Strict blocks and interior neutral points}

Finite compositions play two distinct roles in the proofs of mesh bounds and regularity.

\begin{Rem}\label{rem:blocks}
The compositions in $(Q)$ give the uniform $O(1/n)$ mesh bound.
The blocks in $\mathcal S_m$ used in Lemma~\ref{lem:survivor} give decay of the Lipschitz constants of prefixes under $(R)$.
The two lengths need not coincide.
Under $(H)$, strict pairs suffice for the latter argument: a nonstrict pair has its inner label in $\mathcal P$, and two distinct labels in $\mathcal P$ form a strict pair.
Here $\mathcal P$ is the set of labels of maps with a neutral fixed endpoint, as defined in the proof of Lemma~\ref{lem:run-bound}.
Under the weaker $(H_\partial)$, strict triples suffice instead.
The condition $\mathcal S_m\ne\emptyset$ alone is insufficient for this argument; the proof requires $\mathcal J_m(\omega)$ to be infinite along the relevant address.
Condition $(R)$ gives this recurrence outside a Lebesgue null set; Proposition~\ref{prop:finite} gives the stronger conclusion that the exceptional set is countable under $(F)$.
Neither argument requires a probabilistic assumption on addresses.
\end{Rem}

The next example has strict triples but no strict pairs.

\begin{Exa}\label{ex:triple}
Let $N=2$ and
\[ f_j(x)=x-\frac{1-\cos(\pi x)}4+\frac j2, \qquad j=0,1.\]
These maps are compatible and $C^\infty$, and
\[ f_j'(x)=1-\frac\pi4\sin(\pi x)>0, \qquad E_j =\{0,1\}, \qquad j = 0,1.\]
Thus $(H)$ fails, while $(H_\partial)$ holds. 
Moreover, $f_0''(0)=-\pi^2/4$ and $f_1''(1)=\pi^2/4$, so condition $(ND)$ holds.
Hence Corollary~\ref{cor:C2-estimates} applies. 
Evaluating the derivative at $x=j$ gives, for every $i,j$, $\Lip(f_i\circ f_j)=1$.

On the other hand, $\Lip(f_i\circ f_j\circ f_k)=1$ if and only if $j=k$.
If $j=k$, then $(f_i \circ f_j \circ f_j)^{\prime}(x) = f_i' (f_j(f_j(x))) f_j'(f_j(x)) f_j'(x)$ and hence $(f_i \circ f_j \circ f_j)^{\prime}(j) = f_i' (f_j(f_j(j))) f_j'(f_j(j)) f_j'(j) = f_i'(j) = 1$, $j = 0,1$. 
Conversely, if $\Lip(f_i\circ f_j\circ f_k)=1$, then there exists $x$ such that $ f_i' (f_j(f_k(x))) = f_j'(f_k(x)) = f_k'(x) = 1$.
Since $E_0 = E_1 =\{0,1\}$, $j=k=x$. 
Thus a triple is strict whenever its last two symbols differ, whereas no pair is a strict contraction.

In this example, $(Q)$ holds with $q=1$.
For $t=y-x$,
\begin{align*}
 f_j(y)-f_j(x)
 &=t-\frac12\sin\!\left(\frac{\pi(x+y)}2\right) \sin\!\left(\frac{\pi t}2\right)\\
 &\le t-\frac12\sin^2\!\left(\frac{\pi t}2\right) \le t-\frac{t^2}2\le\frac{t}{1+t/2}.
\end{align*}
Here $(x+y)/2\in[t/2,1-t/2]$ and $\sin(\pi t/2)\ge t$ for $0\le t\le1$.
\end{Exa}

The following example has an interior neutral fixed point.

\begin{Exa}\label{ex:interior}
Let $N=3$ and
\[ f_0(x)=\frac x8,\qquad f_1(x)=x-\frac12\left(x-\frac12\right) \left|x-\frac12\right|,\qquad f_2(x)=\frac78+\frac x8.\]
The images of these maps are $[0,1/8]$, $[1/8,7/8]$, and $[7/8,1]$. 
We have
\[ f_1'(x)=1-\left|x-\frac12\right|, \qquad f_1(1/2)=1/2,\qquad f_1'(1/2)=1.\]
Thus $E_0=E_2=\emptyset$ and $E_1=\{1/2\}$.
For $t=y-x$,
\[ f_1(y)-f_1(x) =t-\int_x^y\left|u-\frac12\right|\,du \le t-\frac{t^2}4\le\frac{t}{1+t/4}.\]
The affine maps $f_0$ and $f_2$ satisfy the same estimate. 
Thus $(F)$ holds with $p=1$, and $(Q)$ holds with $q=1$ and $a=1/4$.
Consequently, Proposition~\ref{prop:finite} and Corollary~\ref{cor:block} apply, although $(H_\partial)$ fails and $f_1$ is not $C^2$. 
The pointwise conclusion holds outside $D_f\cup\{f_\sigma(1/2):\sigma\in\A^*\}$.
For affine target maps $\{g_i\}_i$ given by $g_i(x)=p_ix+\sum_{j<i}p_j$ with $p_i>0$ and $\sum_ip_i=1$, the conclusion applies to the corresponding balanced IFS measures.
\end{Exa}

Figure~\ref{fig:source-examples} compares the source maps $\{f_i\}_i$ in Examples~\ref{ex:triple} and~\ref{ex:interior}.

\begin{figure}[htbp]
\centering
\includegraphics[width=\linewidth]{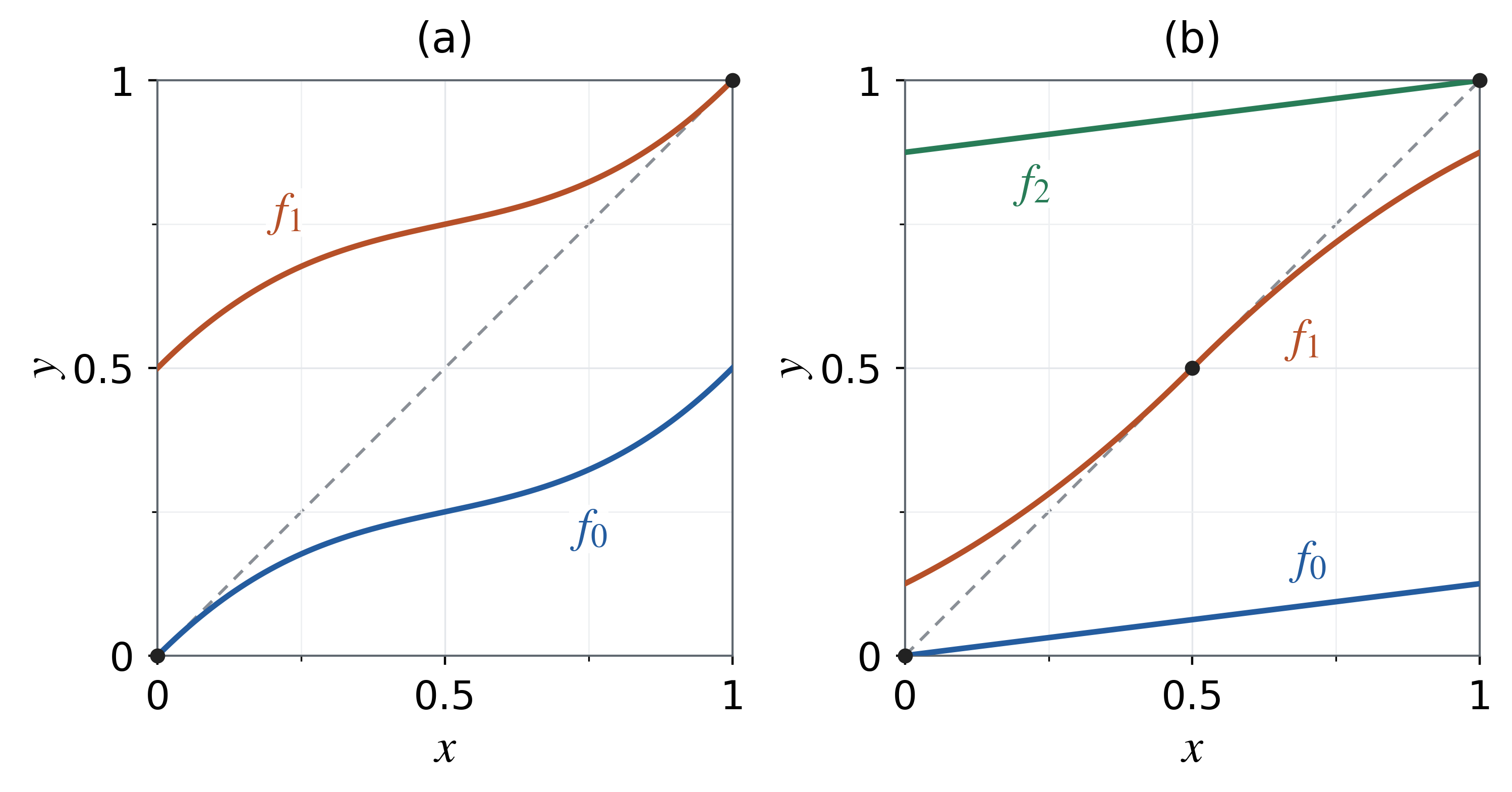}
\caption{Source maps $\{f_i\}_i$ in Example~\ref{ex:triple} (a) and Example~\ref{ex:interior} (b).
The dashed line is $y=x$, and the marked points are the fixed points of $\{f_i\}_i$.}
\label{fig:source-examples}
\end{figure}

\subsection{Null survivors without finite equality sets}

Condition $(R)$ can hold even when $(F)$ fails.

\begin{Exa}\label{ex:R-not-F}
Condition~$(R)$ is strictly weaker than~$(F)$, even among compatible $C^1$ systems satisfying~$(M)$.
More precisely, there exists such a system for which $K_1=\{0\}$, whereas, for every $p\ge1$, the set $\{x:(f_0^p)'(x)=1\}$ contains a nondegenerate interval.
Its conjugacy measure is UST regular for every compatible target system $\{g_i\}_i$ satisfying~\eqref{eq:basic-g}.
\end{Exa}

\begin{proof}
Let $h(s) \coloneqq \left(\max\{-\sin(2\pi s),0\}\right)^2$ and define
\[ d(x) \coloneqq \int_0^x t\,h(\log_2t)\,dt\quad x>0, \qquad d(0) \coloneqq 0, \qquad \theta \coloneqq d(1),\]
\[ f_0(x) \coloneqq x-d(x),\qquad f_1(x) \coloneqq 1-\theta+\theta x.\]
The integrand extends continuously to $0$ with value $0$.
Consequently, $d\in C^1([0,1])$, $d'(0)=0$, and
\[ f_0'(x)=1-xh(\log_2x), \quad x>0, \qquad f_0'(0)=1.\]
Since $0\le h\le1$, we have $f_0'(x)\ge1-x>0$ for $0<x<1$.
Moreover, $f_0'(1)=1$, because $h(0)=0$.
Thus $0<f_0'\le1$ on $[0,1]$.
Also, $0<\theta\le1/2<1$, so $f_1'=\theta$.
In particular, the maps are strictly increasing and $1$-Lipschitz.
They are compatible because
\[ f_0(0)=0,\qquad f_0(1)=1-\theta=f_1(0),\qquad f_1(1)=1.\]

Figure~\ref{fig:null-survivors} shows the source maps $\{f_i\}_i$ and the function $1-f_0'$.

\begin{figure}[htbp]
\centering
\includegraphics[width=\linewidth]{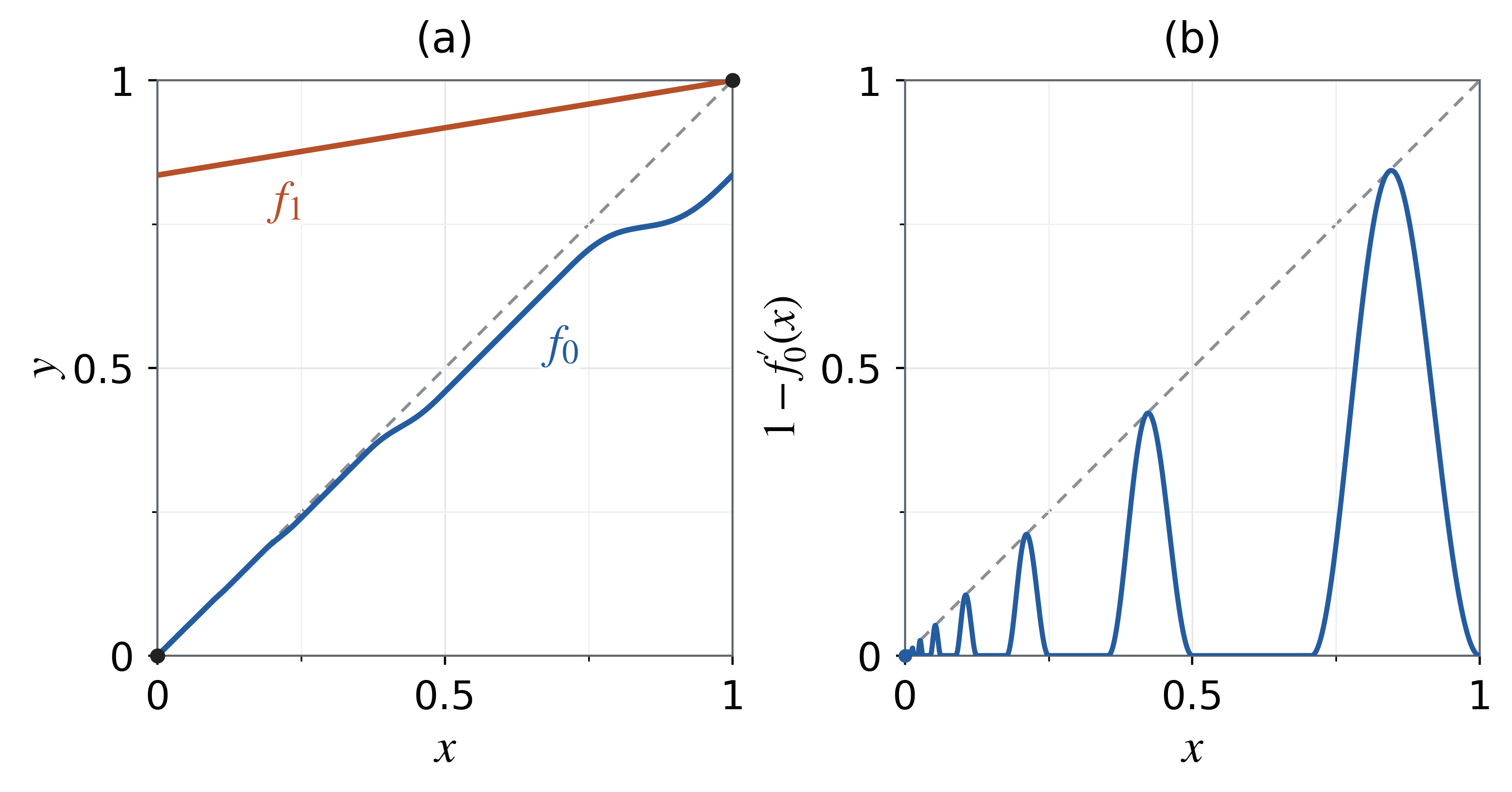}
\caption{The source maps $\{f_i\}_i$ in Example~\ref{ex:R-not-F} (a) and the function $1-f_0'(x)=xh(\log_2x)$, extended continuously by $0$ at $x=0$ (b).
The dashed line is $y=x$ in both panels.
The zero intervals in (b), on which $f_0'=1$, accumulate at $0$.}
\label{fig:null-survivors}
\end{figure}

Since the function $h$ is $1$-periodic, a change of variables gives
\[ d\left(\frac{x}{2}\right)=\frac14  d(x), \qquad 0\le x\le1.\]
Moreover, $d(x)>0$ for every $x>0$.
The positive continuous function $d(x)/x^2$ has a positive minimum and a finite maximum on $[1/2,1]$.
The scaling identity $\frac{d(x/2)}{(x/2)^2} = \frac{d(x)}{x^2}$ extends these bounds to $(0,1]$.
Hence there are $a,b>0$ such that $ ax^2\le d(x)\le bx^2, \ 0\le x\le1$.

Let $s_r \coloneqq f_0^r(1)$.
Since $0<s_{r+1} = f_0 (s_r) = s_r-d(s_r)\le s_r$, we have
\[ \frac1{s_{r+1}}-\frac1{s_r} =\frac{d(s_r)}{s_rs_{r+1}}\ge a.\]
It follows that
\begin{equation}\label{eq:R-not-F-run}
 |f_0^r([0,1])| = s_r\le\frac1{1+ar} \le\frac{C}{r+1}, \qquad r\ge0,
\end{equation}
where $C \coloneqq \max\{1,a^{-1}\}$.

Write any word $\sigma$ of length $n$ in the form
\[ \sigma=0^{r_0}1\,0^{r_1}1\cdots1\,0^{r_k},  \qquad r_j\ge0, \qquad n=k+\sum_{j=0}^k r_j.\]
Here $k$ is the number of occurrences of $1$, and we allow empty runs of $0$, that is, $r_j = 0$.
Thus consecutive occurrences of $1$ are allowed.
For each $j$, the prefix preceding the run $0^{r_j}$ has Lipschitz constant at most $\theta^j$, while the suffix following that run maps $[0,1]$ into itself.
Thus
\[ |f_\sigma([0,1])| \le\frac{C\theta^j}{r_j+1}, \quad 0\le j\le k.\]
Summing these inequalities after multiplication by $r_j+1$ yields
\[ (n+1) |f_\sigma([0,1])|  =\sum_{j=0}^k(r_j+1) |f_\sigma([0,1])|  \le C\sum_{j=0}^k\theta^j  \le\frac{C}{1-\theta}.\]
This proves~$(M)$.

Since $\Lip(f_0)=1$ and $\Lip(f_1)=\theta<1$, the only infinite word all of whose length-one blocks have Lipschitz constant $1$ is $0^\infty$.
Equation~\eqref{eq:R-not-F-run} gives
\[ K_1=\bigcap_{r\ge1}f_0^r([0,1])=\{0\}.\]
Hence~$(R)$ holds.
Proposition~\ref{prop:mesh} gives $(UW)$, and Theorem~\ref{thm:main} gives UST regularity for every target system $\{g_i\}_i$ satisfying~\eqref{eq:basic-g}.

It remains to show that~$(F)$ fails for every block length.
For every integer $n\ge1$, the function $h(\log_2x)$ vanishes on $J_n=[\alpha_n,\beta_n] \coloneqq [2^{-n},2^{-n+1/4}]$.
Thus $d$ is constant on $J_n$, with value $d_n=d(2^{-n})=\theta 4^{-n}$, and $f_0(x)=x-d_n$ there.
Fix $p\ge1$.
For all sufficiently large $n$, the interval $I_{n,p}=[\alpha_n+(p-1)d_n,\beta_n]$ is nondegenerate, since $\beta_n-\alpha_n$ is a positive constant times $2^{-n}$, whereas $d_n=\theta 4^{-n}$.
If $x\in I_{n,p}$, then $f_0^j(x)=x-jd_n\in J_n, \ 0\le j\le p-1$. 
Each derivative factor in the chain rule for $f_0^p$ is therefore $1$, and $I_{n,p}\subset\{x:(f_0^p)'(x)=1\}$.
Thus $(F)$ fails for every $p\ge1$.

Finally, the source system is not $C^2$ at $0$: for $x>0$,
\[ \frac{f_0'(x)-f_0'(0)}{x} = -h(\log_2x),\]
which has no limit as $x\downarrow0$.
\end{proof}

\subsection{Endpoint drift and quadratic contraction bounds}

Endpoint nondegeneracy alone does not imply a quadratic contraction estimate for each map.

\begin{Rem}\label{rem:single}
The assumptions $C^2+(H)+(ND)$ do not imply $(Q)$ with $q=1$.
For example, take
\[ f_0(x) \coloneqq \frac{x+3x^2-x^3}{4},\qquad f_1(x) \coloneqq \frac34+\frac x4.\]
Here $E_0=\{1\}$ and $E_1=\emptyset$. 
Since $f_0'(0)=f_1'(1)=1/4<1$, condition $(ND)$ holds. 
However, $f_0(1)-f_0(1-t)=t-\frac{t^3}{4}$ cannot be bounded above by $t/(1+at)$ for all sufficiently small $t$ for any $a>0$.
On the other hand, all four compositions $f_i\circ f_j$, $i,j\in\{0,1\}$, are strict contractions.
Since there are only finitely many such compositions, $(Q)$ holds with $q=2$.
Thus the value $q=2$ in Proposition~\ref{prop:ND1} cannot in general be replaced by $q=1$.
\end{Rem}

The next example satisfies $(H_\partial)$ and $(D)$, so Proposition~\ref{prop:endpoint} applies.
Nevertheless, $(Q)$ fails for every fixed composition length.

\begin{Exa}\label{ex:oscillatory}
For $x>0$, let
\[ F(x) \coloneqq x^2\left(\frac12 + \frac{2\sin(\log x)-\cos(\log x)}5\right)+\frac{x^3}3, \qquad F(0) \coloneqq 0,\]
and define the $C^1$ maps
\[  f_0(x) \coloneqq x-\frac14F(x),\qquad f_1(x) \coloneqq \frac{101}{120}+\frac{19}{120}x.\]
The maps are compatible since $f_0(1) = 101/120$.
We also have
\[ f_0'(x)=1-\frac{x(1+\sin(\log x))+x^2}{4}, \quad x>0,\qquad f_0'(0)=1.\]
In particular, $1/4\le f_0'(x)<1$ for $x \in (0,1]$, $E_0=\{0\}$, and $E_1=\emptyset$. 
Thus $(H_{\partial})$ holds. 
Since
\[ x-f_0(x)\ge\frac14\left(\frac12-\frac1{\sqrt5}\right)x^2,\]
condition $(D)$ holds, and Proposition~\ref{prop:endpoint} applies.
On the other hand, for $u_n \coloneqq \exp(-\pi/2-2\pi n)$,
\[ \frac{1-f_0'(u_n)}{u_n}=\frac{u_n}{4} \rightarrow 0, \ \  n \to \infty.\]
Thus the first inequality in \eqref{eq:differential-endpoints} fails.
The second derivative for $x>0$ is
\[f_0''(x)=-\frac14\left(1+\sin(\log x)+\cos(\log x)+2x\right),\]
which changes sign infinitely often near $0$ and has no limit there.

In fact, $(Q)$ fails for every fixed block length $q$.
Let $J_n \coloneqq [u_n-u_n^{3/2},u_n+u_n^{3/2}]$.
For each fixed $q$, uniformly for $x\in J_n$ and $0\le j<q$,
\begin{equation}\label{eq:log-composition}
f_0^j(x) = x+O_q(u_n^2), \qquad \log(f_0^j(x)) = \log u_n+O_q(u_n^{1/2}). 
\end{equation}
The first estimate follows from $F(t) = O(t^2)$ by telescoping.
For the second estimate, we use $x/u_n=1+O(u_n^{1/2})$, divide the first estimate by $u_n$, and apply $\log(1+t)=O(t)$ as $t\to0$.

Let 
\[ s_n \coloneqq \log u_n=-\frac{\pi}{2}-2\pi n,  \qquad \Delta_{n,j}(x) \coloneqq \log(f_0^j(x))-s_n.\]
By \eqref{eq:log-composition}, $\Delta_{n,j}(x)=O_q(u_n^{1/2})$, uniformly for $x\in J_n$ and $0\le j<q$.
We apply Taylor's formula to the function $s\mapsto\sin s$ at $s=s_n$, with increment $\Delta$.
Since $\sin s_n=-1$ and $\cos s_n=0$, we have
\[ 1+\sin(s_n+\Delta)  =1-\cos\Delta  =\frac{\Delta^2}{2}+O(\Delta^4), \qquad \Delta\to0.\]
Substituting $\Delta=\Delta_{n,j}(x)$, we obtain
\[ 1+\sin(\log(f_0^j(x)))  =O_q(u_n),\]
uniformly for $x\in J_n$ and $0\le j<q$.

The chain rule therefore yields
\[ 1-(f_0^q)'(x)=O_q(u_n^2),\qquad |J_n|-|f_0^q(J_n)|=O_q(u_n^{7/2})=o(|J_n|^2).\]
On the other hand, $(Q)$ applied to the word $0^q$ would give
\[ |J_n|-|f_0^q(J_n)| \ge\frac{a|J_n|^2}{1+a|J_n|},\]
which is a contradiction.
Under $(H_\partial)$, the endpoint drift condition is therefore strictly weaker than the quadratic contraction estimates for compositions of fixed length.
Indeed, Lemma~\ref{lem:mesh-drift} gives $(Q)\Rightarrow(M) \Rightarrow (D)$ under the standing source assumptions.
This example satisfies the mesh bound $(M)$ but fails $(Q)$ for every composition length.
\end{Exa}

\subsection{Pointwise and uniform mass conditions}

The almost-everywhere mass condition can hold without the uniform mass condition, even for smooth source systems satisfying $(H)$.

\begin{Exa}\label{ex:W-not-UW}
Let $b \coloneqq 1/\sqrt2$ and $\theta \coloneqq 1-b$, and consider the source maps $\{f_i\}_{i = 0,1}$ given by
\[ f_0(x)=\frac{x}{\sqrt{1+x^2}},\quad f_1(x)=b+\theta x.\]
This is a compatible smooth system satisfying $(H)$ and $(R)$, but $f_0''(0)=0$, so $(ND)$ fails.
The conjugacy measures have the following properties.\\
(i) For every target system $\{g_i\}_i$ satisfying \eqref{eq:basic-g}, its conjugacy measure $\mu$ satisfies
\[ \lim_{r\downarrow0}-r\log\mu(\Ball r x)=0
 \qquad\text{for Lebesgue-a.e. } x\in[0,1]. \]
In particular, it satisfies $(W)$ and is UST regular.\\
(ii) If the target maps $\{g_i\}_i$ are affine with weights $p_0,p_1>0$ and $p_0+p_1=1$, then $(UW)$ fails.
\end{Exa}

\begin{proof}
The source maps $\{f_i\}_i$ are smooth and compatible, and
\[ f_0'(x)=(1+x^2)^{-3/2},\qquad f_0''(x)=-3x(1+x^2)^{-5/2},\qquad f_1'(x)=\theta<1.\]
Thus $E_0=\{0\}$ and $E_1=\emptyset$, so $(H)$ holds, whereas $f_0'(0)=1$ and $f_0''(0)=0$ violate $(ND)$.
Vanishing mesh follows from Lemma~\ref{lem:finite}.
Direct iteration gives
\begin{equation}\label{eq:iteration-f0} 
f_0^n(x)=\frac{x}{\sqrt{1+nx^2}}, \qquad x \in [0,1]. 
\end{equation}
By \eqref{eq:iteration-f0}, $\delta_n(f)\ge |f_0^n([0,1])|=(n+1)^{-1/2}$, so $(M)$ fails. 
Since $\Lip(f_1)<1$, the only infinite word with no strict block of length one is $0^\infty$.
Consequently,
\[ K_1=\bigcap_{n\ge1}f_0^n([0,1])=\{0\},\]
which proves $(R)$.

For (i), let $H_k \coloneqq f_0^k\circ f_1$, $k \ge 0$. 
Then $H_k(x) = \frac{b+\theta x}{\sqrt{1+k(b+\theta x)^2}}$, and hence
\[ H_k^{\prime}(x) = \frac{\theta}{(1+k(b + \theta x)^2)^{3/2}}. \]
Thus $\Lip(H_k)\le\theta$, and
\[ \left|\frac{d}{dx}\log H_k'(x)\right| =\frac{3k\theta(b+\theta x)}{1+k(b+\theta x)^2} \le \frac{3\theta}{b}.\]
We also have that 
\[ \frac{H_k^{\prime} (x)}{H_k^{\prime} (y)}  \le \exp\left(\frac{3\theta}{b} |x-y| \right), \quad x, y \in [0,1], \ k \ge 0. \]

This estimate and the chain rule give, for every finite composition $H_{k_1}\circ\cdots\circ H_{k_j}$,
\begin{equation}\label{eq:distortion} 
\frac{\max_{[0,1]}(H_{k_1}\circ\cdots\circ H_{k_j})'}{\min_{[0,1]}(H_{k_1}\circ\cdots\circ H_{k_j})'} \le  \exp\!\left(\frac{3\theta}{b(1-\theta)}\right) \eqqcolon D. 
\end{equation}

Let $x \notin D_f$, and let $\omega(x)$ be its unique original address over $\{0,1\}$.
If $\omega(x)$ contained only finitely many occurrences of $1$, then it would have the form $\sigma 0^\infty$ for some finite word $\sigma$.
Since $0^\infty$ codes $0$, this would give $x=f_\sigma(0)\in D_f$, which is a contradiction.

Thus $\omega(x)$ contains infinitely many occurrences of $1$ and has the unique representation
\[ \omega(x) =0^{k_1(x)}1\,0^{k_2(x)}1\,0^{k_3(x)}1\cdots, \qquad k_j(x) \in \{0,1,2,\ldots\}.\]
Here $k_1(x)$ is the number of zeros before the first occurrence of $1$, and, for $j\ge2$, $k_j(x)$ is the number of zeros between the $(j-1)$st and $j$th occurrences of $1$.
The value $k_j (x) = 0$ is allowed, so the representation includes consecutive occurrences of $1$.
Each block $0^{k_j(x)}1$ corresponds to the map $ f_0^{k_j(x)}\circ f_1 = H_{k_j(x)}$.
We call $(k_1(x),k_2(x),\ldots)$ the induced address of $x$.

Let 
\[ J_j(x) \coloneqq H_{k_1(x)}\circ\cdots\circ H_{k_j(x)}([0,1]), \quad j \ge 1, \qquad J_0(x) \coloneqq [0,1].\]

Let $J \coloneqq H_{k_1}\circ\cdots\circ H_{k_{j-1}} ([0,1])$ for $j \ge 2$. 
By \eqref{eq:distortion}, 
\begin{align*} 
\left|  H_{k_1}\circ\cdots\circ H_{k_{j-1}} \left(H_{\ell}([0,1])\right) \right| &\le \max_{[0,1]}(H_{k_1}\circ\cdots\circ H_{k_{j-1}})' \left|H_{\ell}([0,1])\right| \\
&\le D \min_{[0,1]}(H_{k_1}\circ\cdots\circ H_{k_{j-1}})' \left|H_{\ell}([0,1])\right| \\
&\le D |J|  \left|H_{\ell}([0,1])\right|. 
\end{align*}

By \eqref{eq:iteration-f0}, 
$H_k([0,1]) = \left[\frac1{\sqrt{k+2}},\frac1{\sqrt{k+1}}\right]$, and hence $\bigcup_{k\ge K} H_k([0,1])= \left(0, \frac{1}{\sqrt{K+1}} \right]$.
Summing over $\ell \ge K$, we obtain
\[  \Leb\left(\{x\in J\setminus D_f \colon k_j(x)\ge K\}\right) \le \sum_{\ell \ge K} \left|  H_{k_1}\circ\cdots\circ H_{k_{j-1}} \left(H_{\ell}([0,1])\right) \right|  \le  \frac{D|J|}{\sqrt{K+1}}. \]

Since $ \bigcup_{k\ge0} H_k((0,1]) = (0,1]$, 
we obtain that 
\[ \bigcup_{k_1, \dots, k_{j-1} \ge 0}  H_{k_1}\circ\cdots\circ H_{k_{j-1}} ([0,1]) = (0,1] \] 
by induction on $j \ge 2$. 
These cylinders have pairwise disjoint interiors, so their lengths sum to one. 
Summing over the cylinders, we see that 
\[ \Leb\left(\{x\notin D_f:k_j(x)\ge K\}\right) \le \frac{D}{\sqrt{K+1}}.\] 

Fix $0<\varepsilon<-\log\theta$ and apply the Borel--Cantelli lemma with $K=\lceil \exp(\varepsilon j)\rceil$. 
Then, for almost every $x$, there is $C_x<\infty$ such that the original depth of $J_j(x)$ satisfies
\[ n_j(x) \coloneqq \sum_{\ell=1}^j (k_\ell(x)+1) \le C_x \exp(\varepsilon j), \qquad j\ge1.\]
Hence $\mu(J_j(x))\ge c_g^{n_j(x)}$.
Recall that $|J_j(x)| \le \theta^j$. 
For $0 < r < 1$, choose the least $j$ such that $|J_j(x)|\le r$.
Then $J_j(x)\subset\Ball r x$ and $r<|J_{j-1}(x)|\le\theta^{j-1}$.
Consequently, as $r \downarrow 0$,
\[ 0\le-r\log\mu(\Ball r x) \le r n_j(x)\log(1/c_g) \le C_x\theta^{j-1}\exp(\varepsilon j)\log(1/c_g) \to 0.\]
Here $j\to\infty$ as $r\downarrow0$.
This proves the pointwise assertion and hence $(W)$.
Together with $(R)$, Theorem~\ref{thm:main} gives UST regularity.

For (ii), \eqref{eq:iteration-f0}, \eqref{eq:conjugacy}, and the definition of $g_0$ give
\[\mu\!\left(\left[0,\frac1{\sqrt{n+1}}\right]\right) =\mu([0,f_0^n(1)]) =\varphi(f_0^n(1)) =g_0^n(\varphi(1)) =g_0^n(1) =p_0^n.\]

Let $x_n = r_n \coloneqq 1/(2\sqrt{n+1})$ for $n\ge1$.
Then $0<r_n\le\min\{x_n,1-x_n\}$ and
\[ -r_n\log\mu(\Ball{r_n}{x_n}) = \frac{n\log(1/p_0)}{2\sqrt{n+1}} \rightarrow+\infty, \quad n \to \infty.\]
Condition $(UW)$ would give an upper bound independent of $n$, which is a contradiction.
Thus $(W)$ is strictly weaker than $(UW)$ even for conjugacy measures of smooth finite systems.
\end{proof}

In the next example, the maps $f_0$ and $f_2$ and the image intervals of all three maps are the same as in Example~\ref{ex:interior}.
The displacement at the interior neutral fixed point is now cubic rather than quadratic.
The endpoint conditions $(D)$ and $(ND)$ still hold, but the uniform mesh condition $(M)$ fails.

\begin{Exa}\label{ex:cubic-interior}
Let $N=3$ and
\[ f_0(x)=\frac{x}{8},\qquad f_1(x)=x-\left(x-\frac12\right)^3,\qquad f_2(x)=\frac78+\frac{x}{8}.\]
Then the following hold.\\
(i) The source system is compatible and $C^\infty$, has vanishing mesh, and satisfies $(D)$, $(ND)$, $(F)$ with $p=1$, and $(R)$.
However, $(H_\partial)$ and $(M)$ fail, and $(Q)$ fails for every fixed composition length.\\
(ii) For every target system satisfying \eqref{eq:basic-g}, the conjugacy measure $\mu$ satisfies
\[ \lim_{r\downarrow0}-r\log\mu(\Ball r x)=0 \qquad\text{for Lebesgue a.e.  }x\in[0,1].\]
In particular, $(W)$ holds and $\mu$ is UST regular.\\
(iii) If the target maps are affine with positive weights $p_0,p_1,p_2$, then $(UW)$ fails.
\end{Exa}

\begin{proof}
For (i), let $\xi\coloneqq1/2$.
The maps are compatible and $C^\infty$, and
\[ f_0'=f_2'=\frac18,\qquad f_1'(x)=1-3(x-\xi)^2\in[1/4,1].\]
Thus $E_0=E_2=\emptyset,\qquad E_1=\{\xi\}$. 
Consequently, $(F)$ holds with $p=1$, whereas $(H_\partial)$ fails.
Both $(D)$ and $(ND)$ hold because $f_0'(0)=f_2'(1)=1/8<1$.

Define $u_0 \coloneqq \frac12,\ u_{n+1} \coloneqq u_n-u_n^3, n \ge 0$.
Then $u_n\downarrow0$ and $f_1^n([0,1])=[\xi-u_n,\xi+u_n]$. 
For every interval $I\subset[0,1]$ of length $s$, direct calculation gives
\[ |f_i(I)|\le s-\frac{s^3}{4}, \qquad i\in\{0,1,2\},\]
with equality for $i=1$ when $I$ is centered at $\xi$.
Since $s\mapsto s-s^3/4$ is increasing on $[0,1]$, iteration yields $\delta_n(f)=2u_n$.
Moreover,
\[  u_{n+1}^{-2}-u_n^{-2} =\frac{2-u_n^2}{(1-u_n^2)^2} \rightarrow 2, \quad n \to \infty.\]
Hence
\[ u_n\sim\frac1{\sqrt{2n}}, \qquad  \delta_n(f)\sim\sqrt{\frac2n}.\]
Thus the mesh vanishes, but $(M)$ fails.
Since $(Q)$ implies $(M)$ by Corollary~\ref{cor:block}, condition $(Q)$ also fails for every fixed composition length.

The only infinite word containing no strict block of length one is $1^\infty$.
Therefore
\[ K_1=\bigcap_{n\ge1}f_1^n([0,1])=\{\xi\},\]
which proves $(R)$.

For (ii), we apply the inducing argument from Example~\ref{ex:W-not-UW}.
Define
\[ H_{k,i}\coloneqq f_1^k\circ f_i, \qquad k\ge0,\quad i\in\{0,2\}.\]
These maps satisfy
\[ \Lip(H_{k,i})\le\frac18, \qquad \left\|(\log H_{k,i}')'\right\|_\infty \le \frac{32}{7}.\]
For the second estimate, let $\psi(t) \coloneqq t-t^3$ and use
\[ (\psi^j)'(t)\le\left(\frac{\psi^j(t)}{t}\right)^3, \qquad \sum_{j\ge0}\bigl(\psi^j(t)\bigr)^4 \le \frac47 t^2, \qquad \frac38\le t\le\frac12.\]
The first inequality follows from $1-3t^2\le(1-t^2)^3$, and the second follows by telescoping $t^2-\psi(t)^2\ge(7/4)t^4$.
The chain rule then gives the stated bound.
Consequently, all finite compositions of the induced maps have a common distortion bound.

Outside the countable set
\[ Z\coloneqq D_f\cup \bigcup_{\sigma\in\A^*}f_\sigma(\{\xi\}),\]
every address contains infinitely many symbols in $\{0,2\}$ and has a unique representation
\[ \omega(x)=1^{k_1(x)}i_1(x)\,1^{k_2(x)}i_2(x) \cdots, \qquad k_j(x)\ge0,\quad i_j(x)\in\{0,2\}.\]
Let
\[
 \begin{aligned}
 J_j(x)&\coloneqq
 H_{k_1(x),i_1(x)}\circ\cdots\circ
 H_{k_j(x),i_j(x)}([0,1]),\\
 n_j(x)&\coloneqq\sum_{s=1}^j(k_s(x)+1).
 \end{aligned}
\]
The induced tails satisfy
\[ \bigcup_{\substack{k\ge K\\i\in\{0,2\}}}  H_{k,i}([0,1]) =[\xi-u_K,\xi+u_K]\setminus\{\xi\}.\]
Therefore, their total length is $2u_K = O(K^{-1/2})$.
Together with uniform distortion, this gives
\[ \Leb\left(\{x\notin Z:k_j(x)\ge K\}\right) \le \frac{C}{\sqrt{K+1}}, \qquad j,K\ge1,\]
where $C$ is independent of $j$ and $K$.

Fix $0<\varepsilon<\log8$.
Exactly as in Example~\ref{ex:W-not-UW}, the Borel--Cantelli lemma gives $n_j(x)\le C_x \exp(\varepsilon j)$ for Lebesgue a.e.  $x$.
Also,
\[ |J_j(x)|\le8^{-j}, \qquad  \mu(J_j(x))\ge c_g^{n_j(x)}.\]
The same comparison between radii and induced cylinder lengths as in that example proves the asserted pointwise limit.
Thus $(W)$ holds, and $(R)$ together with Theorem~\ref{thm:main} gives UST regularity.

Finally, for affine target maps, the central cylinders have masses
\[ \mu(\Ball{u_n}{\xi})=\mu(f_1^n([0,1]))=p_1^n.\]
Consequently,
\[ \lim_{n \to \infty} -u_n\log\mu(\Ball{u_n}{\xi})  = \lim_{n \to \infty} nu_n\log(1/p_1) = +\infty.\]
This contradicts $(UW)$ and proves (iii).
\end{proof}

\subsection{Regularity without the mass condition}\label{subsec:regular-without-W}

The following example uses the orthogonal-polynomial characterization of UST regularity.
For a probability measure $\mu$ with support $[0,1]$, let $P_n$ be its monic orthogonal polynomial of degree $n$, so that $P_n$ is orthogonal in $L^2(\mu)$ to every polynomial of degree less than $n$.
Then UST regularity is characterized by
\begin{equation*}
 \mu\in\Reg([0,1])  \quad\Longleftrightarrow\quad \lim_{n\to\infty}\|P_n\|_{L^2(\mu)}^{1/n}=\frac14,
\end{equation*}
where $1/4$ is the logarithmic capacity of $[0,1]$.
Equivalently, the positive leading coefficients of the orthonormal polynomials have $n$th roots tending to $4$.
The polynomial $P_n$ minimizes the $L^2(\mu)$ norm among all monic polynomials of degree $n$.
Indeed, every such polynomial can be written as $Q=P_n+S$, where $\deg S<n$, and orthogonality gives
\[ \|Q\|_{L^2(\mu)}^2 =\|P_n\|_{L^2(\mu)}^2+\|S\|_{L^2(\mu)}^2 \ge\|P_n\|_{L^2(\mu)}^2.\]

For general measures, UST regularity does not imply $(W)$.

\begin{Exa}\label{ex:regular-not-W}
There exists a nonatomic probability measure $\mu$ with support $[0,1]$ such that $\mu\in\Reg([0,1])$, but
\begin{equation}\label{eq:regular-not-W}
 \limsup_{r\downarrow0}\left[-r\log\mu(\Ball r x)\right]=+\infty
 \qquad\text{for Lebesgue a.e. }x\in[0,1].
\end{equation}
Thus $(W)$ is not necessary for UST regularity among general measures.
This example does not determine whether the condition is necessary for conjugacy measures of finite systems satisfying $(R)$.

Let $C\subset[-1/2,1/2]$ be an affine copy of the middle-thirds Cantor set.
Then $C$ is compact and has Lebesgue measure zero.
Let $\mathcal P(C)$ denote the set of Borel probability measures supported on $C$.
For $\rho\in\mathcal P(C)$, define its logarithmic energy by
\[ I(\rho) \coloneqq \iint_{C\times C}\log\frac1{|x-y|}\,d\rho(x)\,d\rho(y).\]
The logarithmic capacity of $C$ is
\[ c\coloneqq\caplog(C)= \exp(-V_C), \qquad  V_C\coloneqq\inf_{\rho\in\mathcal P(C)}I(\rho),\]
with the convention $\exp(-\infty) = 0$.

Let $\eta$ be the probability measure that assigns mass $2^{-k}$ to each basic interval at level $k$ of the Cantor construction. 
Then $I(\eta) < \infty$ and hence $c > 0$.

By \cite[Theorem~A.9]{Simon}, there exists a unique measure $\nu\in\mathcal P(C)$ attaining the minimum energy: $I(\nu)=V_C=-\log c$.
This minimizer is called the equilibrium probability measure of $C$.
The measure $\eta$ above is used only to establish that $C$ has positive capacity.
The construction below uses the equilibrium measure $\nu$.
Since $I(\nu)<\infty$, the measure $\nu$ is nonatomic. 

We use the following global upper bound for the logarithmic potential of $\nu$.
By Frostman's theorem \cite[Theorem~A.10(a)]{Simon},
\begin{equation}\label{eq:Frostman} 
U^\nu(z) \coloneqq \int_C\log\frac1{|z-t|}\,d\nu(t) \le V_C=-\log c,  \qquad z \in \mathbb C. 
\end{equation}

For $n\ge1$, let
\[ H_n(x) \coloneqq T_n(2x-1),\qquad a_n \coloneqq 2^{2n-1}, \qquad C_n \coloneqq H_n^{-1}(C),\]
where $T_n$ is the Chebyshev polynomial of degree $n$ and $a_n$ is the leading coefficient of $H_n$.

For each $n\ge1$, let $x_{n,j}\coloneqq\frac{1-\cos(j\pi/n)}2, \ j=0,\ldots,n$. 
Then
\[ 0=x_{n,0}<x_{n,1}<\cdots<x_{n,n}=1,\qquad H_n(x_{n,j})=(-1)^{n-j}.\]
The map $H_n$ is strictly monotone on each interval $[x_{n,j-1},x_{n,j}]$ and maps it onto $[-1,1]$.
Since $C\subset[-1/2,1/2]\subset(-1,1)$, for every $t\in C$ and every $j=1,\ldots,n$, there is exactly one point $x\in(x_{n,j-1},x_{n,j})$ satisfying $H_n(x)=t$.
Let $\psi_{n,j}\colon(-1,1)\longrightarrow(x_{n,j-1},x_{n,j})$ denote the corresponding inverse map.
Thus $H_n(\psi_{n,j}(t)) = t, \  t\in(-1,1)$.

Let 
\[ \nu_n \coloneqq \frac1n\sum_{j=1}^n (\psi_{n,j})_{*} \nu, \]
where $(\psi_{n,j})_*\nu$ denotes the push-forward of $\nu$ under $\psi_{n,j}$. 
The maps  $\{\psi_{n,j}\}_j$ are smooth on a neighborhood of $C$, so $C_n = \bigcup_{j = 0}^{n} \psi_{n,j}(C)$ is compact and $\Leb(C_n) = 0$, and $\nu_n$ is nonatomic and supported on $C_n$.
The factorization $H_n(z)-t=a_n\prod_{j=1}^n(z-\psi_{n,j}(t))$ and \eqref{eq:Frostman} give
\begin{equation}\label{eq:pullback-potential}
 \begin{aligned}
 U^{\nu_n}(z)&=\frac{U^\nu(H_n(z))+\log a_n}{n} \le-\log c_n,\\
 c_n&=\left(\frac{c}{a_n}\right)^{1/n} = \frac14(2c)^{1/n} \rightarrow \frac14, \ n \to \infty.
 \end{aligned}
\end{equation}

Let $F_n \coloneqq \bigcup_{j=1}^nC_j$.
Since each $F_n$ is compact and Lebesgue null, choose $r_n>0$ with $r_{n+1}<r_n/2$ and
\[ \Leb(\{x\in[0,1] \colon \dist(x,F_n)\le r_n\}) \le 2^{-n}.\]
Define
\[ t_0=1,\qquad t_n \coloneqq e^{-n/r_n},\qquad w_n=t_{n-1}-t_n,\qquad
 \mu=\sum_{n=1}^\infty w_n\nu_n.\]
Then $w_n>0$, $\sum_n w_n=1$, and $\sum_{j>n} w_j=t_n$.
Hence $\mu$ is a nonatomic probability measure. 
We have $\nu_n ((x_{n,j-1},x_{n,j})) = 1/n$.
The lengths of the monotonicity intervals satisfy
\[ x_{n,j}-x_{n,j-1}  = \sin\!\left(\frac{(2j-1)\pi}{2n}\right)  \sin\!\left(\frac{\pi}{2n}\right).\] 
In particular,
\[ \max_{1\le j\le n}(x_{n,j}-x_{n,j-1}) \le \sin\!\left(\frac{\pi}{2n}\right)\le \frac{\pi}{2n} \rightarrow0, \ n \to \infty.\]
Thus the maximum length of these intervals tends to zero. 
Consequently, every nonempty open subinterval of $[0,1]$ has positive $\mu$-mass, and $\operatorname{supp}\mu=[0,1]$.

For any monic polynomial $Q$ of degree $m$, factorization into its complex roots and \eqref{eq:pullback-potential} give
\[ \int_{C_n} \log|Q(x)|\,d\nu_n(x)\ge m\log c_n.\]
By Jensen's inequality, for every fixed $n$ we have
\[ \|Q\|_{L^2(\mu)} \ge\sqrt{w_n}\,\|Q\|_{L^2(\nu_n)} \ge\sqrt{w_n}\,c_n^m.\]

Let $P_m$ be the monic $\mu$-orthogonal polynomial of degree $m$.
Substituting $Q=P_m$ into the preceding inequality gives
\[ \|P_m\|_{L^2(\mu)}^{1/m} \ge w_n^{1/(2m)}c_n.\]
For each fixed $n$, we have $w_n^{1/(2m)}\to1$ as $m\to\infty$, and hence
\[ \liminf_{m\to\infty}\|P_m\|_{L^2(\mu)}^{1/m} \ge c_n.\]
Letting $n\to\infty$ and using \eqref{eq:pullback-potential}, we obtain
\[ \liminf_{m\to\infty}\|P_m\|_{L^2(\mu)}^{1/m}  \ge \frac14.\]

By orthogonality, $P_m$ minimizes the $L^2(\mu)$ norm among all monic polynomials of degree $m$.
Comparison with the monic Chebyshev polynomial $Q_m(x)=2^{1-2m}T_m(2x-1)$ gives
\[ \|P_m\|_{L^2(\mu)} \le \|Q_m\|_{L^2(\mu)} \le \|Q_m\|_{\infty,[0,1]} = 2^{1-2m}.\] 
Consequently,
\[ \limsup_{m\to\infty}\|P_m\|_{L^2(\mu)}^{1/m}  \le \frac{1}{4}.\]

These bounds prove that $\mu\in\Reg([0,1])$.

On the other hand, the Borel--Cantelli lemma gives, for Lebesgue almost every $x$, $\dist(x,F_n)>r_n$ for all sufficiently large indices.
For every such $x$,
\[ \mu(\Ball{r_n}x)\le\sum_{j>n}w_j = \exp(-n/r_n), \qquad -r_n\log\mu(\Ball{r_n}x)\ge n\]
for all sufficiently large $n$.
This proves \eqref{eq:regular-not-W}.
\end{Exa}

\section{Consequences for recurrence coefficients}\label{sec:cesaro-nevai}

The regularity results above also control the recurrence coefficients of orthogonal polynomials in an averaged sense.
Let $\mu$ be a probability measure with support $[0,1]$.
Using the monic orthogonal polynomials from Subsection~\ref{subsec:regular-without-W}, let $p_n \coloneqq P_n/\|P_n\|_{L^2(\mu)}$ for $n\ge0$.
The orthonormal polynomials satisfy
\[ xp_n(x) = \alpha_{n+1}p_{n+1}(x)+\beta_np_n(x)+\alpha_np_{n-1}(x), \qquad n\ge0,\]
where $p_{-1}=0$, $p_0=1$, $\alpha_0=0$, $\alpha_n>0$ for $n\ge1$, and $\beta_n\in\mathbb R$ for $n\ge0$.

We recall the Ces\`aro--Nevai condition in the normalization corresponding to $[0,1]$.
The terminology follows Golinskii and Khrushchev \cite{GolinskiiKhrushchev} for orthogonal polynomials on the unit circle and Simon \cite{SimonCesaro} for those on the real line.

\begin{Def}\label{def:cesaro-nevai}
The measure $\mu$ satisfies the Ces\`aro--Nevai condition if
\[ \lim_{m\to\infty}\frac1m\sum_{n=1}^{m} \left(\left|\alpha_n-\frac14\right|+\left|\beta_n-\frac12\right|\right)=0.\]
\end{Def}

Since the recurrence coefficients are bounded, this condition is equivalent to
\[ \lim_{m\to\infty}\frac1m\sum_{n=1}^{m} \left[ \left(\alpha_n-\frac14\right)^2+ \left(\beta_n-\frac12\right)^2 \right]=0.\]
It is also equivalent to requiring, for every $\varepsilon>0$, that the proportion of indices $1\le n\le m$ for which $|\alpha_n-1/4|+|\beta_n-1/2|>\varepsilon$ tends to zero as $m\to\infty$.
This is a statement about the natural density of polynomial degrees.

The main theorem gives the following consequence for recurrence coefficients.

\begin{Cor}\label{cor:cesaro-nevai}
Under the hypotheses of Theorem~\ref{thm:main}, the measure $\mu$ satisfies the Ces\`aro--Nevai condition in Definition~\ref{def:cesaro-nevai}.
\end{Cor}

\begin{proof}
Theorem~\ref{thm:main} gives $\mu\in\Reg([0,1])$.
Let $\widetilde\mu$ be its image under $T(x)=4x-2$.
Then $\widetilde\mu$ is regular with support $[-2,2]$, and its recurrence coefficients are $\widetilde\alpha_n=4\alpha_n$ and $\widetilde\beta_n=4\beta_n-2$.
Applying \cite[Theorem~1.1]{SimonCesaro} and dividing by $4$ proves the assertion.
\end{proof}

The usual Nevai condition requires $\lim_{n \to \infty} \alpha_n = 1/4$ and $\lim_{n \to \infty} \beta_n = 1/2$.
The Ces\`aro--Nevai condition permits deviations along a set of indices of natural density zero.
Corollary~\ref{cor:cesaro-nevai} does not assert convergence of the individual recurrence coefficients.\\

\noindent{\bf Use of generative AI} \ During the preparation of this work, the author used ChatGPT 6 Astra and Claude Fable 5.1 to assist with exploration, structuring, drafting, and refining the text. 
After using these tools, the author reviewed and edited the content as needed. 
The author takes full responsibility for the content and integrity of the publication. \\

\bibliographystyle{amsplain}
\bibliography{general_UST_regularity_references}
\end{document}